\documentclass[12pt]{amsart}

\usepackage{amsmath,amssymb,amsthm,esint}
\usepackage{geometry}
\usepackage{enumitem}
\makeatletter
\newcommand{\mylabel}[2]{#2\def\@currentlabel{#2}\label{#1}}
\makeatother

\usepackage[pdfpagemode=UseNone,pdfstartview=FitH]{hyperref}

\newcommand{\abs}[1]{\left|#1\right|}
\newcommand{\bdry}[1]{\partial #1}
\newcommand{\bgset}[1]{\big\{#1\big\}}
\newcommand{\closure}[1]{\overline{#1}}
\newcommand{\dint}{\ds{\int}}
\newcommand{\dist}[2]{\text{dist}\, (#1,#2)}
\newcommand{\ds}[1]{\displaystyle #1}
\newcommand{\eps}{\varepsilon}
\newcommand{\goodchi}{\protect\raisebox{2pt}{$\chi$}}

\newcommand{\interior}[1]{#1^\circ}
\newcommand{\ip}[2]{\left<#1,#2\right>}

\newcommand{\norm}[2][]{\left\|#2\right\|_{#1}}
\renewcommand{\O}{\text{O}}
\renewcommand{\o}{\text{o}}

\newcommand{\QED}{\mbox{\qedhere}}
\newcommand{\restr}[2]{\left.#1\right|_{#2}}
\newcommand{\seq}[1]{\left(#1\right)}
\newcommand{\set}[1]{\left\{#1\right\}}
\newcommand{\smset}[1]{\{#1\}}
\newcommand{\strictsubset}{\subset \subset}
\newcommand{\vol}[1]{\L(#1)}

\newcommand{\B}{{\mathcal B}}
\newcommand{\D}{{\mathcal D}}
\renewcommand{\L}{{\mathcal L}}
\newcommand{\M}{{\mathcal M}}
\newcommand{\calN}{{\mathcal N}}

\newcommand{\N}{\mathbb N}
\newcommand{\R}{\mathbb R}

\DeclareMathOperator{\divg}{div}
\DeclareMathOperator{\supp}{supp}

\newtheorem{lemma}{Lemma}[section]
\newtheorem{proposition}[lemma]{Proposition}
\newtheorem{theorem}[lemma]{Theorem}
\newtheorem{corollary}[lemma]{Corollary}

\theoremstyle{definition}
\newtheorem{definition}[lemma]{Definition}

\numberwithin{equation}{section}

\title{Higher critical points
in an elliptic free boundary problem}
\date{}
\author{David Jerison and Kanishka Perera}
\thanks{The first-named author was supported
by NSF grants DMS 1069225, DMS 1500771, and the Stefan Bergman Trust.}
\thanks{This work was initiated while the second-named author
was visiting the Department of Mathematics at the Massachusetts Institute
of Technology, and he is grateful for the kind hospitality of the department.}
\address{David Jerison, Department of Mathematics, Massachusetts Institute of Technology, 77  Massachusetts Avenue, Cambridge, MA 02139.}
\email{jerison@math.mit.edu}
\address{Kanishka Perera, Department of Mathematical Sciences,
Florida Institute of Technology,
Melbourne, FL 32901.}
\email{kperera@fit.edu}

\begin{document}


\begin{abstract}
We study higher critical points of the variational functional
associated with a free boundary
problem related to plasma confinement.
Existence and regularity of minimizers in elliptic free boundary
problems have already been studied extensively.   But because
the functionals are not smooth,
standard variational methods cannot be used directly to prove the
existence of higher critical points. Here we find a nontrivial
critical point of mountain pass type and prove
many of the same estimates known for minimizers, including
Lipschitz continuity and nondegeneracy.  We then show that
the free boundary is smooth in dimension $2$ and prove
partial regularity in higher dimensions.


\end{abstract}

\maketitle

\section{Introduction}\label{sec: intro}

In this paper we consider a superlinear free boundary problem
related to plasma confinement (see, e.g.,
\cite{MR587175,MR1644436,MR1360544,MR1932180,MR0412637,MR0602544}).
Let $\Omega$ be a bounded domain in $\R^N$ with smooth boundary,
and define the functional
\[
J(v) =\int_\Omega \left[\frac{1}{2}\, |\nabla v|^2 + Q_p(x,v)\right] \, dx
\]
with
\[
Q_p(x,v) =  \goodchi_{\set{v > 1}}(x) - \frac{1}{p}\, (v - 1)_+^p
\]
for $2 < p < \infty$ if $N = 2$ and for $2 < p < 2N/(N - 2)$ if $N \ge 3$.
We seek a non-minimizing critical point of this functional in
the usual Sobolev space $H^1_0(\Omega)$, the closure of $C_0^\infty(\Omega)$
in the norm
\[
\norm{v}^2 = \int_\Omega |\nabla v|^2\, dx.
\]
The critical point $u$ of $J$ that we construct is Lipschitz continuous
in $\bar \Omega$.  The region
\[
\set{u>1} \subset\subset \Omega
\]
represents the plasma, and the boundary of the plasma,
\[
F(u) : = \partial \{x\in \Omega:  u(x) >1\},
\]
is the free boundary.

The function $u$ satisfies the following interior Euler-Lagrange equation
\begin{equation}\label{eq:interior}
-\Delta u  = (u-1)_+^p, \quad \mbox{in} \quad \Omega \setminus{F(u)},
\end{equation}
where $w_\pm = \max \set{\pm w,0}$ denote the positive and negative parts of $w$,
respectively.
The function $u$ also satisfies, in various generalized forms,
the free boundary condition
\[
|\nabla u^+|^2 - |\nabla u^-|^2  = 2 \quad \text{on }\  F(u),
\]
where $\nabla u^\pm$ are the limits of $\nabla u$ from
$\set{u > 1}$ and $\interior{\set{u \le 1}}$, respectively.
The ultimate goal is to show that at most (or all) points, the free
boundary is smoooth, and at those points the free boundary condition
is satisfied in the ordinary, classical sense.

The assumption $p>2$ makes the Euler-Lagrange  equation
superlinear, which helps us to prove existence of a nontrivial mountain
pass solution.  We also make use of the assumption $p>2$
in proving important nondegeneracy properties of $u$ that
lead to regularity of the free boundary.
The upper limitation on $p$ is imposed so that the inclusion from
$H^1_0(\Omega)$ to $L^p(\Omega)$ is a compact.  The limiting
exponent $p  = 2N/(N-2)$, $N\ge 3$, is treated in \cite{yangyang-perera}.


Our first theorem, Theorem \ref{Theorem 1}, says that there is
a Lipschitz continuous mountain pass solution to the variational problem.
Our second theorem, Theorem \ref{thm:nondegeneracy},
says that this solution is nondegenerate and satisfies
the free boundary condition in the sense of viscosity.
Our third theorem, Theorem \ref{thm:corollary}, establishes
full regularity of the free boundary in dimension $2$
and partial regularity in higher dimensions.
We believe that these are the first results in the literature
to address existence and regularity of higher critical points of free boundary
functionals.  This paper is an improvement on our preprint
\cite{Jerison-Perera}, which established weaker partial regularity
of the free boundary.

For minimizers there is a large literature proving existence
and partial regularity of the free boundary.
(See, for example,
\cite{MR618549,MR732100,MR1044809,MR2145284,MR861482,MR990856,MR1029856,MR973745,MR1620644,MR1759450}
and the references therein). Our results are less general
than those for minimizers, which apply to
many more classes of potentials than $Q_p(x,v)$.  We chose this family
of potential functions because we are able to prove that the corresponding
functional has a nontrivial mountain pass solution.
In addition to being less general, our results give
less regularity for the free boundary than is valid
for minimizers. We have only proved
that our critical point has a smooth free boundary in dimension
$2$.  We conjecture that our results are best possible
in the sense that there does exist an axisymmetric
mountain pass solution in dimension $3$ with a singular free boundary point
resembling the example in \cite{MR618549}.
In the case of minimizers, the best results to date are that the
free boundary is
smooth everywhere in all dimensions $N\le 4$
and has singularities on a closed set of Hausdorff dimension at most $N-5$ in
higher dimensions (see \cite{MR2082392,Jerison-Savin,MR2572253}).


To formulate our results more precisely, we recall the definition
of a mountain pass point.
\begin{definition}[Hofer \cite{MR812787}]
We say that $u \in H^1_0(\Omega)$ is a mountain pass point of $J$ if the set $\set{v \in U : J(v) < J(u)}$ is neither empty nor path connected for every neighborhood $U$ of $u$.
\end{definition}
Let
\[
\Gamma = \set{\gamma \in C([0,1],H^1_0(\Omega)) : \gamma(0) = 0,\, J(\gamma(1)) < 0}
\]
be the class of continuous paths from $0$ to the
set $\set{u \in H^1_0(\Omega) : J(u) < 0}$, and denote
\[
c^* = c^*(\Omega)  := \inf_{\gamma \in \Gamma}\, \max_{u \in \gamma([0,1])}\, J(u).
\]
It will follow from an integration by parts that our mountain pass point
$u$ belongs to the Nehari-type manifold
\[
\M = \set{u \in H^1_0(\Omega) : \int_{\set{u > 1}} |\nabla u|^2\, dx = \int_{\set{u > 1}} (u - 1)^p\, dx > 0}.
\]
Our first main result is the following.
\begin{theorem} \label{Theorem 1}  Let $\Omega$ be a smooth bounded
domain in $\R^N,\, N \ge 2$, and $J$ as above.
Then

a) $c^* = c^*(\Omega) >0$.

b) The functional $J$ has a mountain pass point $u$ satisfying
$J(u) = c^*$, and $u$ minimizes $\restr{J}{\M}$. In particular,
by part (a) the solution is nontrivial.

c)  The function $u$ is Lipschitz continuous on $\bar \Omega$ solving
the interior Euler-Lagrange equation \eqref{eq:interior}.
Moreover, $u$ solves the free boundary condition in the variational
sense of Definition \ref{def:variational}.
\end{theorem}

The following nondegeneracy is the fundamental estimate needed to
be able to establish more detailed properties of the
free boundary.
\begin{definition} \label{def:nondegenerate}
We say that a continuous function $u$ in $\bar \Omega$ is {\em nondegenerate}
if there exist constants $r_0,\, c > 0$ such that if $x_0 \in \set{u > 1}$ and $r := \dist{x_0}{\set{u \le 1}} \le r_0$, then $u(x_0) \ge 1 + cr$.
\end{definition}

Our other main results are as follows.
\begin{theorem}\label{thm:nondegeneracy}  The mountain pass solution
 $u$ in Theorem \ref{Theorem 1} is
nondegenerate in the sense of Definition \ref{def:nondegenerate}
and satisfies the free boundary condition in the sense of  viscosity,
namely, if there is a ball $B$ tangent to the free boundary
and a point $x_0\in \bdry \set{u>1} \cap \bdry B$,
then $u$ has an asymptotic expansion of the form
\[
u(x) =  \alpha \langle x-x_0, \nu\rangle_+ - \beta
\langle x-x_0, \nu\rangle_- + o(|x-x_0|), \quad x\to x_0,
\]
with
\[
\alpha >0, \quad \beta \ge 0, \quad \alpha^2 - \beta^2 = 2,
\]
where $\nu$ is the interior unit normal to $\bdry B$ at $x_0$
if $B\subset \set{u>1}$ and the exterior unit normal
if $B \subset \interior{\set{u\le 1}}$.
\end{theorem}

\begin{theorem} \label{thm:corollary} The mountain pass solution $u$ in
Theorem \ref{Theorem 1} has a free boundary $\partial \{ u>1\}$ of
finite $(N-1)$-dimensional Hausdorff measure that is a $C^\infty$ hypersurface
except on a closed set of Hausdorff dimension at most $N-3$.
Near the smooth subset of the free boundary, $(u-1)_\pm$ are
smooth and the free boundary equation is satisfied in the classical sense.
If $N=2$, then the exceptional set is empty, that is, the free boundary
is smooth at every point. In dimension $N=3$, the free boundary
has at most finitely many nonsmooth points.
\end{theorem}

The proof of Theorem \ref{thm:corollary} depends on two propositions
of independent interest.  Define
\[
\delta_0 := \dist{\set{u > 1}}{\bdry{\Omega}} > 0.
\]
\begin{proposition} \label{Theorem 2}
If $u$ is a nondegenerate, Lipschitz continuous interior solution
as in \eqref{eq:interior},
then there exists a constant $C > 0$ such that whenever $r \le \delta_0/2$,
\[
\sigma(\bdry{\set{u > 1}} \cap B_r(x_0)) \le C r^{N-1},
\]
where $\sigma$ denotes $(N - 1)$-dimensional Hausdorff measure.
In particular, the free boundary $\bdry{\set{u > 1}}$ has
finite $(N - 1)$-dimensional Hausdorff measure.
\end{proposition}
\begin{proposition} \label{Theorem 3}
If $u$ is a nondegenerate, Lipschitz continuous interior solution
as in \eqref{eq:interior} that
minimizes $\restr{J}{\M}$, then there is a constant
$c > 0$ such that whenever $x_0 \in \bdry{\set{u > 1}}$ and $0 < r \le \delta_0/2$,
\begin{equation} \label{55}
c \le \frac{\vol{\set{u > 1} \cap B_r(x_0)}}{\vol{B_r(x_0)}} \le 1 - c,
\end{equation}
where $\L$ denotes the Lebesgue measure in $\R^N$.  Thus,
the topological boundary of $\set{u>1}$ coincides with
its measure-theoretic boundary.
\end{proposition}
\noindent

Let us point out that the existence of a mountain pass solution is by no
means routine due to the lack of smoothness of $J$.  Indeed,
$J$ is not even continuous, much less of class $C^1$.  For
the functional in which the discontinuous term $\chi_{\set{u > 1}}$ is removed,
there is no difficulty in applying the mountain pass theorem, as, for
example, in Flucher and Wei \cite{MR1644436} and Shibata \cite{MR1932180}.

The outline of the proof of Theorem \ref{Theorem 1} is as follows.
In Section 2, we introduce
an approximation $J_\epsilon$ to the functional $J$,
find associated mountain pass solutions $u^\epsilon$,
and prove uniform Lipschitz bounds on these solutions with
the help of a uniform estimate of Caffarelli, Jerison, and Kenig
\cite{MR1906591} (see Proposition \ref{Proposition 3}).
Along the way, we show that $c^*>0$ (part (a) of the theorem).
In Section 3, we show that a subsequence of $u^\epsilon$ converges
to a function $u$ that solves the Euler-Lagrange equation in the complement
of the free boundary.
In Section 4 we show that our putative solution $u$ belongs to the Nehari manifold
$\M$ and minimizes $J$ when restricted to $\M$.  We also show
that $J(u) = c^*$, which ultimately leads to the variational
equation for $u$.

In Section 5 we prove Theorem \ref{thm:nondegeneracy}
by showing that any Lipschitz continuous minimizer of $J$ on $\M$
solving the interior equation \eqref{eq:interior} is nondegenerate.
For minimizers, nondegeneracy is proved using a harmonic replacement.
Our proof of nondegeneracy is somewhat different; it depends on $p>2$ and
projection onto the Nehari manifold.
The second part of the theorem is a corollary of theorems of Lederman and
Wolanski \cite{MR2281453}, which say
that if a singular limit $u$ such as ours is nondegenerate, then
it is a viscosity solution. (Note, however,
that we obtain a stronger form of viscosity solution
because of a further complementary nondegeneracy proved
in Proposition \ref{Theorem 3}.)

In Section 6 we prove Proposition \ref{Theorem 2}, and
in Section 7 we prove Proposition \ref{Theorem 3}.
Both bounds in Proposition \ref{Theorem 3} should be  viewed
as nondegeneracy estimates.
The lower bound by $c$ is an easy consequence of
the nondegeneracy of Definition \ref{def:nondegenerate}.
The upper bound by $1-c$ is a new kind of complementary
nondegeneracy of the region $\{ u\le 1 \}$.
In Section 8, we conclude the proof of
Theorem \ref{thm:corollary} using a blow-up argument based on the
monotonicity formula of G. Weiss described in the appendix, Section 9.

\section{Approximate mountain pass solutions} \label{sec: mountain pass}


We approximate $J$ by $C^1$-functionals as follows. Let $\beta : \R
\to [0,2]$ be a smooth function such that $\beta(t) = 0$ for $t \le
0$, $\beta(t) > 0$ for $0 < t < 1$, $\beta(t) = 0$ for $t \ge 1$, and
$\int_0^1 \beta(s)\, ds = 1$. Then set
\[
\B(t) = \int_0^t \beta(s)\, ds,
\]
and note that $\B : \R \to [0,1]$ is a smooth nondecreasing function
such that $\B(t) = 0$ for $t \le 0$, $\B(t) > 0$ for $0 < t < 1$, and
$\B(t) = 1$ for $t \ge 1$. For $\eps > 0$, let
\[
J_\eps(u) = \int_\Omega \left[\frac{1}{2}\, |\nabla u|^2 + \B\left(\frac{u - 1}{\eps}\right) - \frac{1}{p}\, (u - 1)_+^p\right] dx, \quad u \in H^1_0(\Omega)
\]
and note that $J_\eps$ is of class $C^1$.

If $u$ is a critical point of $J_\eps$, then $u$ is a weak solution of
\begin{equation} \label{2}
\left\{\begin{aligned}
\Delta u & = \frac{1}{\eps}\, \beta\left(\frac{u - 1}{\eps}\right) - (u - 1)_+^{p-1} && \text{in } \Omega\\[10pt]
u & = 0 && \text{on } \bdry{\Omega},
\end{aligned}\right.
\end{equation}
and hence also a classical $C^{2,\alpha}$ solution by elliptic regularity theory.

Note that if $u$ is not identically zero, then it is nontrivial in a stronger
sense, namely, $u>0$ in  $\Omega$ and $\{ u>1 \}$ is a nonempty open set.  In
fact, if $u\le 1$ then it is harmonic in $\Omega$ and hence identically
zero (since $u=0$ on $\partial \Omega$).  Thus any nonzero $u$ is strictly
greater than $1$ on an open set.  Furthermore, on $\{u<1\}$, $u$ is
the harmonic function with boundary values $0$ on $\partial \Omega$ and $1$
on $\partial \{ u\ge 1\}$, hence strictly positive. (Here we are using
the assumption that $\Omega$ is connected.)

Let $\Phi\in C_0^1(\Omega, \R^N)$.  Then by \eqref{2},
\begin{align*}
\mbox{div}\, & \left[ \left( \frac12 |\nabla u|^2
+ \B((u-1)/\eps) - \frac1p (u-1)_+^p\right) \Phi
- (\nabla u \cdot \Phi)\nabla u\right] \\
& = \left( \frac12 |\nabla u|^2
+ \B((u-1)/\eps) - \frac1p (u-1)_+^p\right) \mbox{div}\, \Phi
- \nabla u(D\Phi)\cdot \nabla u \, .
\end{align*}
Hence,
\begin{equation} \label{eq:eps.variation}
\int_\Omega
\left[\left( \frac12 |\nabla u|^2
+ \B((u-1)/\eps) - \frac1p (u-1)_+^p\right) \mbox{div}\, \Phi
- \nabla u(D\Phi)\cdot \nabla u \right] \, dx = 0 \, .
\end{equation}
This is one form of the critical equation that we will ultimately show is
inherited in the limit as $\eps \to 0$ by our mountain pass solution.
It is the critical point equation for $J_\eps$ with
respect to domain variations.  Indeed, for sufficiently small $t$,
$x\mapsto x + t \Phi(x)$ is a diffeomorphism of $\Omega$, and
the left side of \eqref{eq:eps.variation} is
\[
\left. \frac{d}{d t} \right|_{t = 0} J_\eps (u(x+ t \Phi(x))) \, .
\]

\begin{lemma}
$J_\eps$ satisfies the Palais-Smale compactness condition, that is,
every sequence $\seq{u_j} \subset H^1_0(\Omega)$ such that
$J_\eps(u_j)$ is bounded and $J_\eps'(u_j) \to 0$ in $H^1_0(\Omega)$
norm has a convergent subsequence.
\end{lemma}

\begin{proof}
We have
\begin{equation} \label{3.1}
J_\eps(u_j) = \int_\Omega \left[\frac{1}{2}\, |\nabla u_j|^2 + \B\left(\frac{u_j - 1}{\eps}\right) - \frac{1}{p}\, (u_j - 1)_+^p\right] dx = \O(1)
\end{equation}
and
\begin{multline} \label{3.2}
J_\eps'(u_j)\, v_j
= \int_\Omega \left[\nabla u_j \cdot \nabla v_j +
\frac{1}{\eps}\, \beta\left(\frac{u_j - 1}{\eps}\right) v_j
- (u_j - 1)_+^{p-1}\, v_j\right] dx = \o(\norm{v_j}),\\[7.5pt]
v_j \in H^1_0(\Omega).
\end{multline}
We begin by showing that $\norm{u_j}$ is bounded.
Write $u_j = u_j^+ + u_j^-$, where $u_j^\pm$ are defined by
\[
u_j^+ := (u_j - 1)_+, \quad u_j^- = 1 - (u_j - 1)_- \, .
\]
Since $\B$ is bounded \eqref{3.1} gives
\[
\int_\Omega \left[|\nabla u_j^+|^2 + |\nabla u_j^-|^2 - \frac{2}{p}\, (u_j^+)^p\right] dx \le C < \infty.
\]
Taking $v_j = u_j^+$ in \eqref{3.2} and using
\[
\int_\Omega |v| \, dx \le C\|v\|
\]
and the fact that $\beta$ is bounded, we have
\[
\int_\Omega (u_j^+)^p \, dx \le \int_\Omega |\nabla u_j^+|^2 \, dx +
 C \|u_j^+\|.
\]
Combining our inequalities gives
\[
\left(1 - \frac{2}{p}\right) \|u_j^+\|^2 + \|u_j^-\|^2 \le C \left(\|u_j^+\| + 1\right),
\]
which implies that $\|u_j^\pm\|$ are bounded, and hence $\norm{u_j}$, is bounded.

Replace $u_j$ by a subsequence (still denoted $u_j$)
that tends weakly to $u$ in $H_0^1(\Omega)$
and such that $u_j$ tends to $u$ in $L^p(\Omega)$ norm and pointwise
almost everywhere.  Then $J_\eps'(u_j)u\to 0$ and $J_\eps'(u_j)u_j \to 0$
imply that
\[
\lim_{j\to \infty} \|u_j\|^2 = \|u\|^2.
\]
Finally,
\[
\limsup_{j\to \infty} \|u_j-u\|^2
= \limsup_{j\to \infty} (\|u_j\|^2 + \|u\|^2- 2\langle u_j, u \rangle) =
2\|u\|^2 - 2\langle u, u \rangle = 0
\]
\end{proof}

Since $p < 2N/(N - 2)$, the Sobolev imbedding theorem implies
\[
J_\eps(u) \ge \int_\Omega \left[\frac{1}{2}\, |\nabla u|^2 - \frac{1}{p}\, |u|^p\right] dx \ge \frac{1}{2} \norm{u}^2 - C \norm{u}^p \quad \forall u \in H^1_0(\Omega)
\]
for some constant $C$ depending on $\Omega$.
Since $p > 2$, then there exists a constant $\rho > 0$ such that
\[
\norm{u} \le \rho \implies J_\eps(u) \ge \frac{1}{3} \norm{u}^2.
\]
Moreover,
\[
J_\eps(u) \le \int_\Omega \left[\frac{1}{2}\, |\nabla u|^2 + 1 - \frac{1}{p}\, (u - 1)_+^p\right] dx
\]
and hence, again because $p>2$, there exists
a function $u_0 \in H^1_0(\Omega)$ such that $J_\eps(u_0) < 0 = J_\eps(0)$.
Therefore, the class of paths
\[
\Gamma_\eps = \set{\gamma \in C([0,1],H^1_0(\Omega)) : \gamma(0) = 0,\, J_\eps(\gamma(1)) < 0}
\]
is nonempty and
\begin{equation} \label{22}
c_\eps := \inf_{\gamma \in \Gamma_\eps}\, \max_{u \in \gamma([0,1])}\, J_\eps(u) \ge \frac{\rho^2}{3}.
\end{equation}

\begin{lemma} \label{Lemma 7}
$J_\eps$ has a (nontrivial) critical point $u^\eps$ at the level $c_\eps$.
\end{lemma}

\begin{proof}
If not, then there exists a constant $0 < \delta \le c_\eps/2$ and a
continuous map $\eta : \set{J_\eps \le c_\eps + \delta} \to
\set{J_\eps \le c_\eps - \delta}$ such that $\eta$ is the identity on
$\set{J_\eps \le 0}$ by the first deformation lemma (see, e.g., Perera
and Schechter \cite[Lemma 1.3.3]{MR3012848}). By the definition of
$c_\eps$, there exists a path $\gamma \in \Gamma_\eps$ such that
$\max_{\gamma([0,1])}\, J_\eps \le c_\eps + \delta$. Then
$\widetilde{\gamma} := \eta \circ \gamma \in \Gamma_\eps$ and
$\max_{\widetilde{\gamma}([0,1])}\, J_\eps \le c_\eps - \delta$,
contradicting the definition of $c_\eps$.
\end{proof}

\begin{lemma} \label{Lemma 1}
We have $c_\eps \le c^*$. In particular, by \eqref{22}, $c^*>0$ and Theorem \ref{Theorem 1}
(a) holds.
\end{lemma}

\begin{proof}
Since $\B((t - 1)/\eps) \le \goodchi_{\set{t > 1}}$ for all $t$, $J_\eps(u) \le J(u)$ for all $u \in H^1_0(\Omega)$. So $\Gamma \subset \Gamma_\eps$ and
\[
c_\eps \le \max_{u \in \gamma([0,1])}\, J_\eps(u) \le \max_{u \in \gamma([0,1])}\, J(u) \quad \forall \gamma \in \Gamma. \QED
\]
\end{proof}

For $0 < \eps \le 1$, $u^\eps$ have the following uniform regularity properties.

\begin{lemma} \label{Lemma 3}
There exists a constant $C > 0$ such that, for $0 < \eps \le 1$, $\norm{u^\eps} \le C$.
\end{lemma}

\begin{proof}
By Lemma \ref{Lemma 1},
\[
\int_\Omega \left[\frac{1}{2}\, |\nabla u^\eps|^2 + \B\left(\frac{u^\eps - 1}{\eps}\right) - \frac{1}{p}\, (u^\eps - 1)_+^p\right] dx \le c^*
\]
and hence
\begin{equation} \label{32}
\frac{1}{2} \int_\Omega |\nabla u^\eps|^2\, dx \le c^* + \frac{1}{p} \int_{\set{v_\eps > 0}} v_\eps^p\, dx,
\end{equation}
where $v_\eps = u^\eps - 1$. Testing \eqref{2} with $(u^\eps - 1 - \eps)_+$ gives
\begin{equation} \label{33}
\int_{\set{u^\eps > 1 + \eps}} |\nabla u^\eps|^2\, dx = \int_{\set{v_\eps > \eps}} v_\eps^{p-1}\, (v_\eps - \eps)\, dx.
\end{equation}
Fix $\lambda > 2/(p - 2)$. Multiplying \eqref{33} by $(\lambda + 1)/p \lambda$ and subtracting from \eqref{32} gives
\begin{multline} \label{34}
\frac{1}{2} \int_{\set{u^\eps \le 1 + \eps}} |\nabla u^\eps|^2\, dx + \frac{(p - 2) \lambda - 2}{2p \lambda} \int_{\set{u^\eps > 1 + \eps}} |\nabla u^\eps|^2\, dx\\[10pt]
\le c^* + \frac{1}{p} \int_{\set{0 < v_\eps \le \eps}} v_\eps^p\, dx + \frac{1}{p \lambda} \int_{\set{v_\eps > \eps}} v_\eps^{p-1}\, \big[(\lambda + 1)\, \eps - v_\eps\big]\, dx.
\end{multline}
The last integral is less than or equal to $\int_{\set{\eps < v_\eps < (\lambda + 1)\, \eps}} v_\eps^{p-1}\, \big[(\lambda + 1)\, \eps - v_\eps\big]\, dx$ and hence \eqref{34} gives
\[
\min \set{\frac{1}{2},\frac{(p - 2) \lambda - 2}{2p \lambda}} \int_\Omega |\nabla u^\eps|^2\, dx \le c^* + \frac{\eps^p\, \vol{\Omega}}{p} \left[1 + (\lambda + 1)^{p-1}\right].
\]
The conclusion follows.
\end{proof}

\begin{lemma} \label{Lemma 2}
There exists a constant $C > 0$ such that, for $0 < \eps \le 1$,
\[
\max_{x \in \Omega}\, u^\eps(x) \le C.
\]
\end{lemma}

\begin{proof}
Since $p < 2N/(N - 2)$, we have $N(p - 2)/2 < 2N/(N - 2)$. Fix $N(p - 2)/2 < q < 2N/(N - 2)$. Since
\[
- \Delta u^\eps = (u^\eps - 1)_+^{p-1} - \frac{1}{\eps}\, \beta\left(\frac{u^\eps - 1}{\eps}\right) \le (u^\eps-1)_+^{p-1} \le  (u^\eps)^{p-1},
\]
there exists a constant $C > 0$ such that
\[
\norm[\infty]{u^\eps} \le C \norm[q]{u^\eps}^{2q/(2q - N(p-2))}
\]
by Bonforte et al. \cite[Theorem 3.1]{BonforteGrilloVazquez}. Since $u^\eps$ is bounded in $L^q(\Omega)$ by the Sobolev imbedding theorem and Lemma \ref{Lemma 3}, the conclusion follows.
\end{proof}

By Lemma \ref{Lemma 2}, $(u^\eps - 1)_+^{p-1} \le A_0$ for some constant $A_0 > 0$
independent of $\eps$.   Let $\varphi_0 > 0$ be the solution of
\[
\left\{\begin{aligned}
- \Delta \varphi_0 & = A_0 && \text{in } \Omega\\[10pt]
\varphi_0 & = 0 && \text{on } \bdry{\Omega}.
\end{aligned}\right.
\]

\begin{lemma} \label{Lemma 4}
For $0 < \eps \le 1$,
\[
u^\eps(x) \le \varphi_0(x) \quad \forall x \in \Omega,
\]
in particular, $\set{u^\eps \ge 1} \subset \set{\varphi_0 \ge 1} \strictsubset \Omega$.
\end{lemma}

\begin{proof}
Since $\beta(t) \ge 0$ for all $t$,
\[
- \Delta u^\eps \le (u^\eps - 1)_+^{p-1} \le A_0 = - \Delta \varphi_0,
\]
so $u^\eps \le \varphi_0$ by the maximum principle.
\end{proof}

\begin{lemma} \label{Lemma 5} There exists a constant $C > 0$ such
that, for $r>0$ and $0 < \eps \le 1$, if $B_r(x_0) \subset\Omega$, then
\[
\max_{x \in B_{r/2}(x_0)}\, |\nabla u^\eps(x)| \le C/r.
\]
\end{lemma}

\begin{proof}
Since $\beta(t) \le 2$ for all $t$,
\[
\Delta u^\eps \le \frac{1}{\eps}\, \beta\left(\frac{u^\eps - 1}{\eps}\right) \le \frac{2}{\eps}\, \goodchi_{\set{|u^\eps - 1| < \eps}}(x),
\]
and since $\beta(t) \ge 0$ for all $t$,
\[
- \Delta u^\eps \le (u^\eps - 1)_+^{p-1} \le A_0,
\]
so
\[
\pm \Delta u^\eps \le \max \set{2,A_0} \left(\frac{1}{\eps}\, \goodchi_{\set{|u^\eps - 1| < \eps}}(x) + 1\right).
\]
Since $u^\eps$ is also uniformly bounded in $L^2(\Omega)$ by
Lemma \ref{Lemma 2}, the conclusion follows from the following
result of Caffarelli, Jerison, and Kenig \cite{MR1906591}.
\end{proof}

\begin{proposition}[{\cite[Theorem 5.1]{MR1906591}}] \label{Proposition 3}
Suppose that $u$ is a Lipschitz continuous function on $B_1(0) \subset \R^N$ satisfying the distributional inequalities
\[
\pm \Delta u \le A \left(\frac{1}{\eps}\, \goodchi_{\set{|u - 1| < \eps}}(x) + 1\right)
\]
for some constants $A > 0$ and $0 < \eps \le 1$. Then there exists a constant $C > 0$, depending on $N$, $A$, and $\int_{B_1(0)} u^2\, dx$, but not on $\eps$, such that
\[
\max_{x \in B_{1/2}(0)}\, |\nabla u(x)| \le C.
\]
\end{proposition}

\section{Limits of mountain pass solutions} \label{sec: generalized}


Let $\eps_j \searrow 0$, let $u_j = u^{\eps_j}$ be the critical point of $J_{\eps_j}$
obtained in Lemma \ref{Lemma 7}, and let $c_j=  J_{\eps_j}(u_j)$ (an
abuse of notation, since this value was previously denoted $c_{\eps_j}$).
\begin{lemma} \label{Lemma 6}
There exists a  Lipschitz continuous function $u$ on $\bar \Omega$ such
that $u \in H^1_0(\Omega) \cap C^2(\bar \Omega \setminus \bdry{\set{u > 1}})$,
and, for a suitable sequence $\eps_j$,
\begin{enumerate}[label=\textnormal{(\arabic*)}]
\item[\mylabel{6.i}{(a)}] $u_j \to u$ uniformly on $\closure{\Omega}$,
\item[\mylabel{6.iii}{(b)}] $\displaystyle - \Delta u = (u - 1)_+^{p-1}$
on $\Omega \setminus \bdry{\set{u > 1}}$,
\item [\mylabel{6.ii}{(c)}] $u_j \to u$ strongly in $H^1_0(\Omega)$,
\item [\mylabel{6.iv}{(d)}]
$ J(u) \le \liminf c_j \le \limsup c_j \le J(u) + \vol{\set{u = 1}}$.
\end{enumerate}
Moreover, $u$ is nontrivial in the sense that $J(u) + \L(\set{u = 1}) >0$.
\end{lemma}

\begin{proof}
First we prove \ref{6.i}.
The majorant $\varphi_0$ of Lemma \ref{Lemma 4}
gives a uniform lower bound  $\delta_0>0$ on the distance
from $\{u^\eps \ge 1\}$ to $\partial \Omega$.
Thus $u^\eps$ is positive, harmonic and bounded
by $1$  in a $\delta_0$ neighborhood of $\partial \Omega$.  It follows
from standard boundary regularity theory that $u^\eps$ is uniformly
bounded in a $\delta_0/2$ neighborhood in, say,  $C^3$ norm.  In
particular, the family is compact in $C^2$ norm on this set.
By Lemmas \ref{Lemma 2} and \ref{Lemma 5},
the family $u^\eps$ is uniformly  Lipschitz continuous
on the compact subset of $\Omega$ at distance greater or equal to
$\delta_0/2$ from $\partial \Omega$.  Finally, by Lemma \ref{Lemma 3},
$u^\eps$ is uniformly bounded in $H^1_0(\Omega)$.
Thus we can choose $\eps_j$
so that $u_j$ converges uniformly in $\bar \Omega$ to
a Lipschitz function $u$, and so that there is strong
converence in $C^2$ on a $\delta_0/2$ neighborhood of $\partial \Omega$
and, finally, that there is weak convergence of $u_j$ to $u$ in
$H^1_0(\Omega)$.

Next we show that $u$ satisfies the interior part of the
Euler-Lagrange equation:
\[
- \Delta u = (u - 1)_+^{p-1} \quad \mbox{in}\quad \set{u \ne 1}.
\]
Let $\varphi \in C^\infty_0(\set{u >   1})$.
Then $u \ge 1 + 2\, \eps$ on the support of $\varphi$ for some $\eps > 0$. For all sufficiently large $j$, $\eps_j < \eps$ and $|u_j
- u| < \eps$ by \ref{6.i}. Then $u_j \ge 1 + \eps_j$ on the support of
$\varphi$, so testing
\begin{equation} \label{7}
- \Delta u_j = (u_j - 1)_+^{p-1} - \frac{1}{\eps_j}\, \beta\left(\frac{u_j - 1}{\eps_j}\right)
\end{equation}
with $\varphi$ gives
\[
\int_\Omega \nabla u_j \cdot \nabla \varphi\, dx = \int_\Omega (u_j - 1)^{p-1}\, \varphi\, dx.
\]
Passing to the limit gives
\begin{equation} \label{8}
\int_\Omega \nabla u \cdot \nabla \varphi\, dx = \int_\Omega (u - 1)^{p-1}\, \varphi\, dx
\end{equation}
since $u_j$ converges to $u$ weakly in $H^1_0(\Omega)$ and uniformly
on $\Omega$. Hence $u$ is a distributional (and thus a classical) solution
of $- \Delta u = (u - 1)^{p-1}$ in $\set{u > 1}$.

A similar argument shows that $u$ satisfies $-\Delta u = 0$ in $\set{u < 1}$.
We show next that $u$ is also harmonic in the possibly larger set
$\interior{\set{u \le 1}}$. Since $\beta\ge0$, testing \eqref{7} with any
nonnegative $\varphi \in C^\infty_0(\Omega)$ and passing to the limit gives
\begin{equation} \label{19}
- \Delta u \le (u-1)_+^{p-1} \quad \text{in } \Omega
\end{equation}
in the distributional sense.  On the other hand, since
$u$ is harmonic in $\set{u < 1}$, $\min (u,1)$ satisfies the
super-mean value property.  This implies
(see, for instance, \cite[Remark 4.2]{MR618549})
\[
\Delta \min(u,1) \le 0
\]
in the distributional sense.  Combining with \eqref{19}, we find that
\[
\Delta u = 0
\]
as a distribution on the open set $\interior{\set {u\le 1}}$.
Thus the same equation holds in the strong sense, and
this concludes the proof of \ref{6.iii}.   (Note that we
do not exclude the case of connected components of $\interior{\set {u\le 1}}$
on which $u\equiv 1$.)

Since $u_j$
tends weakly to $u$ in $H^1_0(\Omega)$, $\norm{u} \le
\liminf \norm{u_j}$.  So to prove \ref{6.ii}, it suffices to show
that $\limsup \norm{u_j} \le \norm{u}$. Recall
that $u_j$ converges in $C^2$ norm to $u$ in a neighborhood
of $\partial \Omega$ in $\bar \Omega$.
Let $n$ denote the outer unit normal to $\partial \Omega$.
Multiplying \eqref{7} by $u_j - 1$, integrating by parts, and
noting that $\beta((t - 1)/\eps_j)\, (t - 1) \ge 0$ for all $t$ gives
\begin{equation} \label{9}
\int_\Omega |\nabla u_j|^2\, dx \le \int_\Omega (u_j - 1)_+^p\, dx - \int_{\bdry{\Omega}} \frac{\partial u_j}{\partial n}\, d\sigma \to \int_\Omega (u - 1)_+^p\, dx - \int_{\bdry{\Omega}} \frac{\partial u}{\partial n}\, d\sigma.
\end{equation}
Fix $0 < \eps < 1$. Taking $\varphi = (u - 1 - \eps)_+$ in \eqref{8} yields
\begin{equation} \label{10}
\int_{\set{u > 1 + \eps}} |\nabla u|^2\, dx = \int_\Omega (u - 1)_+^{p-1}\, (u - 1 - \eps)_+\, dx,
\end{equation}
and integrating $(u - 1 + \eps)_-\, \Delta u = 0$ over $\Omega$ gives
\begin{equation} \label{11}
\int_{\set{u < 1 - \eps}} |\nabla u|^2\, dx = - (1 - \eps) \int_{\bdry{\Omega}} \frac{\partial u}{\partial n}\, d\sigma.
\end{equation}
Adding \eqref{10} and \eqref{11}, and letting $\eps \searrow 0$, we find
that\footnote{Here we are using the well known fact that
$\displaystyle \int_{\{u=1\}} |\nabla u|^2 \, dx = 0$.}
\[
\int_\Omega |\nabla u|^2\, dx = \int_\Omega (u - 1)_+^p\, dx - \int_{\bdry{\Omega}} \frac{\partial u}{\partial n}\, d\sigma.
\]
This together with \eqref{9} gives
\[
\limsup \int_\Omega |\nabla u_j|^2\, dx \le \int_\Omega |\nabla u|^2\, dx
\]
as desired.

To prove \ref{6.iv}, write
\begin{multline*}
J_{\eps_j}(u_j) = \int_\Omega \left[\frac{1}{2}\, |\nabla u_j|^2 + \B\left(\frac{u_j - 1}{\eps_j}\right) \goodchi_{\set{u \ne 1}}(x) - \frac{1}{p}\, (u_j - 1)_+^p\right] dx\\[10pt]
+ \int_{\set{u = 1}} \B\left(\frac{u_j - 1}{\eps_j}\right) dx.
\end{multline*}
Since $\B((u_j - 1)/\eps_j)\, \goodchi_{\set{u \ne 1}}$ converges pointwise to $\goodchi_{\set{u > 1}}$ and is bounded by $1$, the first integral converges to $J(u)$ by \ref{6.i} and \ref{6.ii}. Since $0 \le \B(t) \le 1$ for all $t$,
\[
0 \le \int_{\set{u = 1}} \B\left(\frac{u_j - 1}{\eps_j}\right) dx \le \vol{\set{u = 1}}.
\]
\ref{6.iv} follows.

By \ref{6.iv} and \eqref{22},
\[
J(u) + \vol{\set{u = 1}} \ge \frac{\rho^2}{3} > 0
\]
and hence $u$ is nontrivial.
\end{proof}

\section{Critical points on the Nehari manifold} \label{sec: Nehari nondegeneracy}

\begin{lemma}\label{lem: gen.Nehari}  Every nonzero $v\in C^0(\bar \Omega) \cap H_0^1(\Omega)$ that solves $-\Delta v = (v-1)_+^{p-1}$ in
$\Omega \setminus \partial \set{v>1}$
belongs to the Nehari manifold $\M$ and satisfies
$J(v) >0$.
\end{lemma}
\begin{proof}
As before for $u^\eps$, if $v\le 1$ in $\Omega$, then it
is harmonic and hence identically zero.  Thus if $v$ is
not identically zero, $\set{v>1}$ is a nonempty open set, where it
it satisfies $- \Delta v = (v - 1)^{p-1}$.
As in the proof of Lemma \ref{Lemma 6}  \ref{6.ii},
multiplying this equation by $v - 1$ and integrating
over the set $\set{v > 1}$ shows that $v$ lies on $\M$. Finally,
since $v\in \M$,
\[
J(v) = \frac{1}{2} \int_{\set{v < 1}} |\nabla v|^2\, dx + \left(\frac{1}{2} - \frac{1}{p}\right) \int_{\set{v > 1}} |\nabla v|^2\, dx + \vol{\set{v > 1}} > 0. \QED
\]
\end{proof}

\begin{proposition} \label{Proposition 9}
We have
\begin{equation} \label{31}
c^* \le \inf_{v \in \M}\, J(v).
\end{equation}
If $v \in \M$ and $J(v) = c^*$, then $v$ is a mountain pass point of $J$.
\end{proposition}
\begin{proof}

For $v\in H_0^1(\Omega)$, set
\[
v^+ = (v - 1)_+, \quad v^- = 1 - (v - 1)_-; \qquad v = v^- + v^+.
\]
Let
\[
W =\set{v \in H^1_0(\Omega) : v^+ \neq 0, \ \mbox{and} \  v^- \neq 0}
\]
Then $\M \subset W$, and for $v\in W$, we define
the curve
\[
\zeta_v(s) = \begin{cases}
(1 + s)\, v^-, & s \in [-1,0]\\[5pt]
v^- + s\, v^+, & s \in (0,\infty),
\end{cases}
\]
which passes through $v$ at $s = 1$. For $s \in [-1,0]$,
\[
J(\zeta_v(s)) = \frac{(1 + s)^2}{2} \int_{\set{v < 1}} |\nabla v|^2\, dx
\]
is increasing in $s$.  There is a discontinuity in $J$ at $s=0$:
\[
\lim_{s \searrow 0}\, J(\zeta_v(s)) = J(\zeta_v(0)) + \vol{\set{v > 1}} > J(\zeta_v(0)).
\]
For $s \in (0,\infty)$,
\begin{multline} \label{30}
J(\zeta_v(s)) = \frac{1}{2} \int_{\set{v < 1}} |\nabla v|^2\, dx + \frac{s^2}{2} \int_{\set{v > 1}} |\nabla v|^2\, dx - \frac{s^p}{p} \int_{\set{v > 1}} (v - 1)^p\, dx\\[7.5pt]
+ \vol{\set{v > 1}}
\end{multline}
and
\[
\frac{d}{ds}\, J(\zeta_v(s)) = s \left[\int_{\set{v > 1}} |\nabla v|^2\, dx - s^{p-2} \int_{\set{v > 1}} (v - 1)^p\, dx\right].
\]
Define
\begin{equation} \label{71}
s_v = \left[\frac{\dint_{\set{v > 1}} |\nabla v|^2\, dx}{\dint_{\set{v > 1}} (v - 1)^p\, dx}\right]^{1/(p-2)}.
\end{equation}
Thus, we see that $J(\zeta_v(s))$ increases for $s \in [-1,s_v)$,
attains its maximum at $s = s_v$, decreases for $s \in (s_v,\infty)$, and
\begin{equation} \label{26}
\lim_{s \to \infty}\, J(\zeta_v(s)) = - \infty.
\end{equation}

For each $v \in \M$, \eqref{26} implies that
we may choose $\bar{s} > 1$ sufficiently large
so that $J(\zeta_v(\bar{s})) < 0$.  Note
that $s_v=1$.  Therefore,
\[
\gamma_v(t) = \zeta_v((\bar{s} + 1)\, t - 1), \quad t \in [0,1]
\]
defines a path $\gamma_v \in \Gamma$ such that
\[
\max_{w \in \gamma_v([0,1])}\, J(w) = J(\zeta_v(s_v)) = J(v),
\]
so $c^* \le J(v)$. Thus \eqref{31} holds.

Next, suppose $v\in \M$ and $J(v) = c^*$.  Let $U$ be a neighborhood of $v$. The
path $\gamma_v$ passes through $v$ at $t = 2/(\bar{s} + 1) =: \bar{t}$ and
$J(\gamma_v(t)) < c$ for $t \ne \bar{t}$. By the continuity of $\gamma_v$,
there exist $0 < t^- < \bar{t} < t^+ < 1$ such that $\gamma_v(t^\pm) \in
U$, in particular, the set $\set{w \in U : J(w) < c}$ is nonempty. If
it is path connected, then this set contains a path $\eta$ joining
$\gamma_v(t^\pm)$, and reparametrizing $\restr{\gamma_v}{[0,t^-]} \cup
\eta \cup \restr{\gamma_u}{[t^+,1]}$ gives a path in $\Gamma$ on which
$J < c^*$, contradicting the definition of $c^*$. So the set is not path
connected, and $v$ is a mountain pass point of $J$.  This concludes
the proof of Proposition \ref{Proposition 9}.
\end{proof}

We can now conclude the proof of part (b) of Theorem \ref{Theorem 1}.
The limit $u$ obtained in Lemma \ref{Lemma 6} belongs to $\M$ by
Lemma \ref{Lemma 6} \ref{6.iii} and Lemma \ref{lem: gen.Nehari}.
Hence
\[
\inf_\M\, J \le J(u).
\]
By Lemma \ref{Lemma 6} \ref{6.iv}, Lemma \ref{Lemma 1}, and
\eqref{31}, we also have
\[
J(u) \le \liminf c_j \le \limsup c_j \le c^* \le \inf_\M\, J.
\]
In all,
\[
J(u) = c^* = \inf_\M\, J.
\]
Thus, $u$ minimizes $J$ restricted to $\M$, and by Proposition
\ref{Proposition 9} it is a mountain pass point of $J$.
By construction $u$ is Lipschitz continuous on $\bar \Omega$.

The inequalities of the preceding paragraph also show
that
\[
\lim_{j\to \infty} c_j = c^*.
\]
This property will enable us to take the limit in the variational
equations for $u_j$ to show  that
$u$ is a variational solution in the following sense.
\begin{definition} \label{def:variational}
A {\em variational solution} $u$ of the Euler-Lagrange equation
for $J$ is a function $u \in H_0^1(\Omega)$ satisfying
\[
\int_\Omega
\left[\left( \frac12 |\nabla u|^2
+ \chi_{\set{u>1}} - \frac1p (u-1)_+^p\right) \mbox{div}\, \Phi
- \nabla u(D\Phi)\cdot \nabla u \right] \, dx = 0 \, .
\]
for every $\Phi\in C_0^1(\Omega,\R^N)$.
\end{definition}

Note first that $c_j \to c^*$ implies
$J_{\eps_j} (u_j) \to J(u)$ as $j\to \infty$.
Since $u_j$ converges to $u$ uniformly
and strongly in $H_0^1(\Omega)$, we obtain
\[
\lim_{j\to \infty} \int_\Omega \B\left(\frac{u_j-1}{\eps_j}\right)
\, dx = \L(\set{u>1}).
\]
Hence,
\[
\limsup_{j\to\infty}
\int_{\set{u\le 1}} \B\left(\frac{u_j-1}{\eps_j}\right) \, dx
\le
\L(\set{u>1}) -
\liminf_{j\to \infty} \int_{\set{u>1}}
\B\left(\frac{u_j-1}{\eps_j}\right) \, dx
\]
On the other hand, because $u_j$ tends uniformly to $u$,
\[
\liminf_{j\to \infty} \int_{\set{u>1}}
\B\left(\frac{u_j-1}{\eps_j}\right) \, dx
\ge \L(\set{u \ge 1+\delta})
\]
for every $\delta>0$.  Taking the limit as $\delta\to 0$,
we find that
\[
\liminf_{j\to \infty} \int_{\set{u>1}}
\B\left(\frac{u_j-1}{\eps_j}\right) \, dx
\ge \L(\set{u > 1})
\]
Therefore,
\begin{equation} \label{eq:negsetcontrol}
\limsup_{j\to\infty}
\int_{\set{u\le 1}} \B\left(\frac{u_j-1}{\eps_j}\right) \, dx = 0
\end{equation}
It follows from this and the dominated convergence theorem that
\begin{align*}
\lim_{j\to\infty}
\int_\Omega \B\left(\frac{u_j-1}{\eps_j}\right)\mbox{div}\, \Phi \, dx
& =
\lim_{j\to\infty}
\int_{\set{u>1}} \B\left(\frac{u_j-1}{\eps_j}\right)\mbox{div}\, \Phi \, dx  \\
& = \int_{\set{u>1}} \mbox{div}\, \Phi \, dx
\end{align*}
This limiting value takes care of the only potentially discontinuous
term in the variational equation.   The others tend to the
appropriate limits because $u_j$ tends uniformly to $u$ and strongly in $H_0^1(\Omega)$.
Thus the variational equation for $u$ holds because it is the limit of
the variational equation \eqref{eq:eps.variation} for $u_j$.
This concludes the proof of part (c) of Theorem \ref{Theorem 1}.

\section{Nondegeneracy} \label{sec:nondegeneracy}

In this section we prove our main estimate of nondegeneracy.
\begin{proposition} \label{Proposition 2}
If $u$ is a  Lipschitz continuous minimizer of $\restr{J}{\M}$
that satisfies the interior Euler-Lagrange equation \eqref{eq:interior},
then $u$ is nondegenerate as in  Definition \ref{def:nondegenerate}.
\end{proposition}

\begin{proof} For $v \in W$, $\zeta_v$ intersects $\M$ exactly at one point, namely,
where $s = s_v$, and $s_v = 1$ if $v \in \M$. So we can define a
continuous projection $\pi : W \to \M$ by
\[
\pi(v) = \zeta_v(s_v) = v^- + s_v\, v^+.
\]

\begin{lemma} \label{Lemma 8}
For $v \in W$,
\[
J(\pi(v)) = \frac{1}{2} \int_{\set{v < 1}} |\nabla v|^2\, dx + \left(\frac{1}{2} - \frac{1}{p}\right) s_v^2 \int_{\set{v > 1}} |\nabla v|^2\, dx + \vol{\set{v > 1}}.
\]
In particular, for $v \in \M$, since $\pi(v) = v$,
\[
J(v) = \frac{1}{2} \int_{\set{v < 1}} |\nabla v|^2\, dx + \left(\frac{1}{2} - \frac{1}{p}\right) \int_{\set{v > 1}} |\nabla v|^2\, dx + \vol{\set{v > 1}}.
\]
\end{lemma}
\begin{proof}
$J(\pi(v))$ is given by \eqref{30} with $s = s_v$, and
\[
s_v^2 \int_{\set{v > 1}} |\nabla v|^2\, dx = s_v^p \int_{\set{v > 1}} (v - 1)^p\, dx. \QED
\]
\end{proof}

Now consider $u$
Suppose that $B_r(x_0) \subset \set{x\in \Omega: u(x)>1}$ and
$\exists\, x_1 \in \partial B_r(x_0)$ such that $u(x_1) = 1$.
Define
\[
v(y) = \frac{1}{r}\, (u(x_0 + ry) - 1).
\]
Our goal is to show that
\[
\alpha : = v(0)\ge c >0.
\]
We begin by observing that
\begin{equation} \label{35}
0 < v(y) = \frac1r (u(x_0+ry) - u(x_1)) \le \frac{L}{r}|x_0 - x_1 +ry|
\le 2L \quad \forall y \in B_1(0),
\end{equation}
where $L$ is the Lipschitz constant of $u$ in $\set{u \ge 1}$, and
\[
- \Delta v = r^p v^{p-1} \quad \text{in } B_1(0).
\]
Define $h$ by
\[
\left\{\begin{aligned}
- \Delta h & = r^p v^{p-1} && \text{in } B_1(0)\\[10pt]
h & = 0 && \text{on } \bdry{B_1(0)}.
\end{aligned}\right.
\]
Then $|h| \le C L^{p-1} r^p$, and applying the Harnack inequality
to $v-h + \max h$, there is a constant $C$ depending on $N$ and
$L$ such that
\[
v(y) \le C\, (\alpha + r^p) \quad \forall y \in B_{2/3}(0).
\]

Take a smooth cutoff function $\psi : B_1(0) \to [0,1]$ such that $\psi = 0$ in $\closure{B_{1/3}(0)}$, $0 < \psi < 1$ in $B_{2/3}(0) \setminus \closure{B_{1/3}(0)}$ and $\psi = 1$ in $B_1(0) \setminus B_{2/3}(0)$, let
\[
w(y) = \begin{cases}
\min \set{v(y),C\, (\alpha + r^p)\, \psi(y)}, & y \in B_{2/3}(0)\\[7.5pt]
v(y), & \mbox{otherwise}, 
\end{cases}
\]
and set $z(x) = 1 + r w((x-x_0)/r)$. Since $u$ is a minimizer of $\restr{J}{\M}$,
\[
J(u) \le J(\pi(z)).
\]
Since $z^- = u^-$, $z = 1$ in $\closure{B_{r/3}(x_0)}$, and $\set{z > 1} = \set{u > 1} \setminus \closure{B_{r/3}(x_0)}$, Lemma \ref{Lemma 8} implies
this inequality can be rewritten as
\[
\left(\frac{1}{2} - \frac{1}{p}\right) \int_{\set{u > 1}} |\nabla u|^2\, dx + \vol{B_{r/3}(x_0)} \le \left(\frac{1}{2} - \frac{1}{p}\right) s_z^2 \int_{\set{u > 1}} |\nabla z|^2\, dx.
\]
Let $y = (x-x_0)/r$ and define
\[
\D := \set{x \in B_{2r/3}(x_0) : v(y) > C\, (\alpha + r^p)\, \psi(y)}.
\]
Because $z = u$ outside $\D$, the last inequality implies
\begin{equation} \label{36}
s_z^2 \int_\D |\nabla z|^2\, dx + \left(s_z^2 - 1\right) \int_{\set{u > 1} \setminus \D} |\nabla u|^2\, dx \ge \frac{2p}{p - 2}\, \vol{B_{1/3}(0)} r^N.
\end{equation}

Since $\set{z > 1} = \set{u > 1} \setminus \closure{B_{r/3}(x_0)}$ and
$z = 1$ in $\closure{B_{r/3}(x_0)}$,
\[
s_z^{p-2} = \frac{\dint_{\set{z > 1}} |\nabla z|^2\, dx}{\dint_{\set{z > 1}} (z - 1)^p\, dx} = \frac{\dint_{\set{u > 1}} |\nabla z|^2\, dx}{\dint_{\set{u > 1}} (z - 1)^p\, dx}.
\]
Since $z = u$ in $\set{u > 1} \setminus \D$, we have
\[
s_z^{p-2} \le
\frac{\dint_{\set{u > 1}} |\nabla u|^2\, dx + \dint_\D |\nabla z|^2\, dx}{\dint_{\set{u > 1}} (u - 1)^p\, dx - \dint_\D (u - 1)^p\, dx}
=
\frac{A_1 + \dint_\D |\nabla z|^2 \, dx}
{A_1 - \dint_\D (u-1)^p \, dx},
\]
where, since $u\in \M$,
\[
A_1  =  \int_{\set{u>1}} |\nabla u|^2 \, dx = \int_{\set{u>1}} (u-1)^p \, dx.
\]


It follows as in \eqref{35} that $0 < u - 1 < 2Lr$ in $\D$, and $\vol{\D} = \O(r^N)$ as $r \to 0$. Therefore
\[
\int_\D (u - 1)^p\, dx = \O(r^{p+N}).
\]
It follows that
\begin{equation} \label{37}
s_z^{p-2} \le 1 + \frac1{A_1}
\dint_\D |\nabla z|^2\, dx + \O(r^{p+N}).
\end{equation}
We have
\begin{equation} \label{28}
\int_\D |\nabla z|^2\, dx = C^2\, (\alpha + r^p)^2\, r^N \int_{\set{y : x \in \D}} |\nabla \psi|^2\, dy.
\end{equation}
The right-hand side is $\O(r^N)$ since $0 < \alpha < 2L$ by \eqref{35}.
So \eqref{37} gives
\[
s_z^2 \le 1 + \frac{2}{(p - 2)A_1}\dint_\D |\nabla z|^2\, dx  + \O(r^{q+N}),
\]
where $q = \min \set{p,N} \ge 2$. Using this estimate in \eqref{36} now gives
\[
\frac{1}{r^N} \int_\D |\nabla z|^2\, dx + \O(r^q) \ge 2\, \vol{B_{1/3}(0)}.
\]
In view of  \eqref{28}, we find that there are
$r_0,\, c > 0$ such that $r \le r_0$
implies $\alpha \ge c$, which was our goal.
\end{proof}

Since the mountain pass solution of Theorem \ref{Theorem 1} satisfies
the hypotheses of Proposition \ref{Proposition 2}, we obtain
the first part of Theorem \ref{thm:nondegeneracy}.
The fact that this solution is a viscosity solution now
follows from results of Lederman and Wolanski.

We will define a weak viscosity solution is as follows.
\begin{definition}\label{def:viscosity}
We say that $u \in C(\Omega)$ satisfies the free boundary condition
\[
|\nabla u^+|^2 - |\nabla u^-|^2 = 2
\]
in the {\em weak viscosity sense} if
whenever there exist a point $x_0 \in \bdry{\set{u > 1}}$, a ball
$B \subset \set{u > 1}$, then either there are $\alpha_1>0$ and
$\alpha_2>0$ such that $\alpha_1^2 \le 2$ and $\alpha_2^2 \le 2$ and
\[
u(x) = 1 + \alpha_1 \langle x-x_0, \nu \rangle_+
+ \alpha_2 \langle x-x_0, \nu \rangle_-
+ o(|x-x_0|), \quad x\to x_0,
\]
with $\nu$ the interior normal to $\bdry B$ at $x_0$,
or else there are $\alpha>0$ and $\beta\ge 0$ such that $\alpha^2 - \beta^2 = 2$
and
\[
u(x) =  1 + \alpha \ip{x - x_0}{\nu}_+ - \beta\ip{x-x_0}{\nu}_- + \o(|x - x_0|),
\quad x\to x_0.
\]
Moreover, if the ball $B\subset \set{u\le 1}$, then the second asymptotic
formula (with $\alpha$ and $\beta$ as above, but with $\nu$ the exterior
normal to $\bdry B$ at $x_0$) holds.
\end{definition}

Denote
\[
f_j(x) = - (u_j(x) - 1)_+^{p-1}, \quad f(x) = - (u(x) - 1)_+^{p-1}, \quad x \in \Omega.
\]
Since $u_j$ converges uniformly to $u$, $f_j$ converges uniformly to $f$.
Therefore, $u_j$ solves an equation of the form \eqref{2} (denoted $E_\eps(f^\eps)$
in the paper of Lederman and Wolanski \cite{MR2281453}).
Since by Proposition \ref{Proposition 2}, $u$ is nondegenerate,
Corollaries 7.1 and 7.2 of \cite{MR2281453} imply that $u$ satisfies
the free boundary condition in the weak viscosity sense.  Furthermore,
Proposition \ref{Theorem 3}, proved below, shows that the case $u>1$ on both
sides of the free boundary (the case of positive $\alpha_1$ and
$\alpha_2$) is ruled out.  This concludes the proof of
Theorem \ref{thm:nondegeneracy}.


\section{The free boundary has finite Hausdorff measure}
\label{sec: Hausdorff}

In this section we prove Proposition \ref{Theorem 2}.
Let $u$ be a Lipschitz, nondegenerate solution to the interior
equation \eqref{eq:interior}.   The outline and most details
are the same as the proof of Theorem 3.4 of Caffarelli-Salsa
\cite{MR2145284}.   The only difference is that
$u-1$ solves an inhomogeneous equation $\Delta (u-1) = (u-1)_+^{p-1}$ in
$\set{u-1>0}$ rather than being harmonic.

\begin{lemma} \label{Lemma 12} (See (3.4) \cite{MR2145284})
There exists a constant $C > 0$ such that whenever $r\le \delta_0/2$,
with $\delta_0 = \mbox{dist} (\set{u>1}, \bdry\Omega)$,
and $\tau > 0$,
\[
\int_{B_r(x_0) \cap \set{|u - 1| < \tau}} |\nabla u|^2\, dx \le C \tau r^{N-1}.
\]
\end{lemma}

\begin{proof}
Denote by $L$ the Lipschitz constant of $u$ on $\bar \Omega$ and by
$M$ the maximum of $u$ over $\bar \Omega$.  (This estimate
does not depend on nondegeneracy.)

For $0 < \eps < \tau$, let
$u_\eps^\tau = \min \set{(u - 1 -   \eps)_+,\tau - \eps}$.
Since $- \Delta u = (u - 1)^{p-1}$ in $\set{u
  > 1}$
and $u_\eps^\tau = 0$ in $\set{u \le 1 + \eps}$, we have
$-u_\eps^\tau\, \Delta u = (u - 1)^{p-1}\, u_\eps^\tau$ in
$\Omega$. Integrating this equation over $B_r(x_0)$ gives
\[
\int_{B_r(x_0)} \nabla u \cdot \nabla u_\eps^\tau\, dx
= \int_{\bdry{B_r(x_0)}} \frac{\partial u}{\partial n}\, u_\eps^\tau\, d\sigma
+ \int_{B_r(x_0)} (u - 1)^{p-1}\, u_\eps^\tau\, dx,
\]
where $n$ is the outward unit normal to $\bdry{B_r(x_0)}$. Since
$u_\eps^\tau = u - 1 - \eps$ in $\set{1 + \eps < u < 1 + \tau}$
and $u_\eps^\tau$ is constant outside this set,
\[
\int_{B_r(x_0)} \nabla u \cdot \nabla u_\eps^\tau\, dx
= \int_{B_r(x_0) \cap \set{1 + \eps < u < 1 + \tau}} |\nabla u|^2\, dx
\to \int_{B_r(x_0) \cap \set{1 < u < 1 + \tau}} |\nabla u|^2\, dx
\]
as $\eps \searrow 0$. We also have
\[
\abs{\int_{\bdry{B_r(x_0)}} \frac{\partial u}{\partial n}\, u_\eps^\tau\, d\sigma}
\le L\tau \sigma(\partial B_r) \le c_N \tau L r^{N-1},
\]
and
\[
\int_{B_r(x_0)} (u - 1)^{p-1}\, u_\eps^\tau\, dx \le
c_N \tau M^{p-1} r^N.
\]
So for a constant $C$ depending only on $L$, $M$ and the diameter of $\Omega$,
\[
\int_{B_r(x_0) \cap \set{1 < u < 1 + \tau}} |\nabla u|^2\, dx \le
C \tau r^{N-1}.
\]
Since $u$ is harmonic in $\set{u < 1}$, a similar argument gives
the same bound for the integral over $\set{1-\tau u < 1}$.
The conclusion follows since $\nabla u = 0$ a.e.\! on the set $\set{u = 1}$.
\end{proof}

\begin{lemma} \label{Lemma 9} (See Lemma 1.10 of \cite{MR2145284})
There exist constants $r_0,\, \lambda > 0$ such that whenever $x_0 \in \set{u > 1}$ and $r := \dist{x_0}{\set{u \le 1}} \le r_0$, there is a point $x_1 \in \bdry{B_r(x_0)}$ satisfying
\[
u(x_1) \ge 1 + (1 + \lambda)\, (u(x_0) - 1).
\]
\end{lemma}
\begin{proof}
Suppose not. Then there are sequences $\lambda_j \searrow 0$ and $x_j \in \set{u > 1}$ with $r_j := \dist{x_j}{\set{u \le 1}} \to 0$ such that
\[
\max_{x \in \bdry{B_{r_j}(x_j)}}\, u(x) < 1 + (1 + \lambda_j)\, (u(x_j) - 1).
\]
Since $u$ is nondegenerate, we may assume that $u(x_j) \ge 1 + c r_j$ for some constant $c > 0$. Noting that $B_{r_j}(x_j) \subset \set{u > 1}$ and $\exists\, x_j' \in \bdry{B_{r_j}(x_j)}$ such that $u(x_j') = 1$, set
\[
v_j(y) = \frac{1}{r_j}\, (u(x_j + r_j y) - 1), \quad y_j = \frac{1}{r_j}\, (x_j' - x_j).
\]
Then $v_j \in C(\closure{B_1(0)}) \cap C^2(B_1(0))$ satisfies
\begin{gather}
\label{38} - \Delta v_j = r_j^p\, v_j^{p-1} \quad \text{in } B_1(0),\\[7.5pt]
\label{39} \max_{y \in \bdry{B_1(0)}}\, v_j(y) < (1 + \lambda_j)\, v_j(0),\\[7.5pt]
\label{40} v_j(0) \ge c, \quad v_j(y_j) = 0.
\end{gather}

We have
\begin{gather*}
0 \le v_j(y)
= \frac{1}{r_j}\, (u(x_j + r_j y) - u(x_j'))
\le \frac{L}{r_j}\, |x_j - x_j' + r_j y| \le 2L
\quad \forall y \in \closure{B_1(0)},\\[7.5pt]
|v_j(y) - v_j(z)| = \frac{1}{r_j}\, |u(x_j + r_j y) - u(x_j + r_j z)|
\le L\, |y - z| \quad \forall y, z \in \closure{B_1(0)},\\[7.5pt]
\int_{B_1(0)} |\nabla v_j(y)|^2\, dy
= r^{-N} \int_{B_{r_j}(x_j)} |\nabla u(x)|^2\, dx \le N \alpha_N L^2,
\end{gather*}
so, for suitable subsequences, $v_j$ converges weakly in $H^1(B_1(0))$ and
uniformly on $\closure{B_1(0)}$ to some Lipschitz continuous function $v$,
and $y_j$ converges to some point $y_0 \in \bdry{B_1(0)}$.
For any $\varphi \in C^\infty_0(B_1(0))$, testing \eqref{38} with $\varphi$ gives
\[
\int_{B_1(0)} \nabla v_j \cdot \nabla \varphi\, dx
= r_j^p \int_{B_1(0)} v_j^{p-1}\, \varphi\, dx,
\]
and passing to the limit as $r_j\to 0$ gives
\[
\int_{B_1(0)} \nabla v \cdot \nabla \varphi\, dx = 0.
\]
So $v$ is harmonic in $B_1(0)$. By \eqref{39},
\[
\max_{y \in \bdry{B_1(0)}}\, v(y) \le v(0),
\]
and hence $v$ is constant by the maximum principle. On the other hand,
\[
v(0) \ge c > 0 = v(y_0)
\]
by \eqref{40}, which is impossible when $v$ is constant.
\end{proof}

The rest of the proof follows \cite{MR2145284} with no change.
From the preceding lemma, a chaining argument carried
out in Theorem 1.9 and Lemma 3.3 of \cite{MR2145284}
gives the following:
\begin{lemma} \label{Lemma 10}
There exist constants $0 < r_0 \le \delta_0$ and $\gamma > 0$ such that whenever $x_0 \in \bdry{\set{u > 1}}$ and $0 < r \le r_0$, there is a point $x \in B_r(x_0) \setminus B_{r/2}(x_0)$ satisfying $u(x) \ge 1 + \gamma r$, in particular,
\[
\sup_{x \in B_r(x_0)}\, u(x) \ge 1 + \gamma r.
\]
\end{lemma}

Next, at the beginning of the proof of Theorem 3.4 \cite{MR2145284},
the following lemma is deduced from Lemma \ref{Lemma 10}:
\begin{lemma} \label{Lemma 11}
There exist constants $0 < r_0 \le \delta_0$ and $\kappa > 0$ such that whenever $x_0 \in \bdry{\set{u > 1}}$ and $0 < r \le r_0$,
\[
\int_{B_r(x_0)} |\nabla u|^2\, dx \ge \kappa r^N.
\]
\end{lemma}
The rest of the proof of Proposition \ref{Theorem 2} is a covering
argument, exactly as in Theorem 3.4 of \cite{MR2145284}.

\section{Nondegeneracy of the non-plasma phase $\set{u\le 1}$}

Now we turn to the proof of Proposition \ref{Theorem 3},
which says that not only $\set{u>1}$ but also
$\set{u\le 1}$ has significant measure
near each topological boundary point.
The measure-theoretic boundary $\partial_* E$ of a measurable set $E\subset \R^N$
is defined as the set of $x\in \R^N$ such that for all $r>0$,
\[
\L(E \cap B_r(x)) >0 \quad \mbox{and} \quad \L(E^c \cap B_r(x)) >0.
\]
Evidently, the measure-theoretic boundary $\partial_* E$
is a subset of the topological
boundary $\partial E$.  The proposition is
a quantitative, scale-invariant estimate showing
that the topological boundary is contained in the measure-theoretic
boundary.

\begin{lemma} \label{Lemma 13}
Let $B= B_1(0)$ be the unit ball in $R^N$.  Let $h\in C(\bar B)$
be a harmonic function in the ball and such that
\[
h(0)>0; \quad |h(x) - h(y)| \le L|x-y|, \quad x, \, y \in \partial B.
\]
Define
\[
\eps = \sigma(\bdry{B} \cap \set{h\le 0}).
\]
There exists a  constant $C > 0$ depending on dimension and $L$
 such that
\[
\int_{\set{h \le 0}} |\nabla h|^2\, dx \le C \left[\frac{\eps}{h(0)}\right]^{1/N}.
\]
\end{lemma}

\begin{proof}  Since $h(0)>0$, there is at least one point of $\partial B$
at which $h$ is positive.  It follows that
\[
h(y) \ge -2L
\]
for all $y\in \partial B$ and hence in all of $B$ by the maximum
principle.

The Poisson integral formula says
\[
h(x) = \frac{1 - |x|^2}{\omega_N} \int_{\bdry{B}} \frac{h(y)}{|x - y|^N}\,
d\sigma(y), \quad x \in B,
\]
in which $\omega_N = \sigma(\partial B)$.
Therefore,
\begin{align*}
\int_{y\in \bdry{B} \cap \set{h > 0}} \frac{h(y)}{|x - y|^N}\, d\sigma(y)
& \ge \frac{1}{2^N} \int_{\bdry{B} \cap \set{h > 0}} h\, d\sigma  \\
& \ge \frac{1}{2^N} \int_{\bdry{B}} h\, d\sigma
 = \frac{\omega_Nh(0)}{2^{N}}.
\end{align*}
For any $\kappa>0$ and any $x \in B_{1 - \kappa}(0)$, we have
\[
\int_{\bdry{B} \cap \set{h \le 0}} \frac{h(y)}{|x - y|^N}\,
 d\sigma(y)
\ge - \frac{2\eps L}{\kappa^N}
\]

Choose
\[
\kappa = 4(\eps L/\omega_N h(0))^{1/N}.
\]
Then for every $x\in B_{1-\kappa}(0)$, $h(x)>0$, and
hence
\[
\set{h \le 0} \subset B \setminus B_{1 - \kappa}(0).
\]

Denoting the spherical part of the gradient by
$\nabla_\theta$, we have $|\nabla_\theta h|\le L$.
Using the expansion of $h$ in spherical harmonics, we have
\[
\int_{\partial B} |\nabla h|^2 d\sigma
\le 2\int_{\partial B} |\nabla_\theta h|^2 d\sigma \le 2L\omega_N
\]
(In fact, the best constant is $N/(N-1)$, achieved by
linear functions $h$.)
Furthermore, since $|\nabla h|^2$ is subharmonic, for all $r<1$,
\[
\int_{\partial B_r} |\nabla h|^2  d\sigma
\le \int_{\partial B} |\nabla h|^2 d\sigma.
\]
Therefore,
\begin{align*}
\int_{\set{h \le 0}} |\nabla h|^2\, dx
& \le \int_{1-\kappa}^1
\int_{\partial B_r} |\nabla h|^2  \,  d\sigma\, dr \\
& \le 2L\omega_N
\int_{1-\kappa}^1 dr = 2L\omega_N \kappa \\
& = c_N L\left(\frac{\eps L}{h(0)}\right)^{1/N}
\end{align*}
\end{proof}

We will now deduce a variant of Lemma \ref{Lemma 10}.
\begin{lemma} \label{lem:nondeg.allr} There exist
constants $r_0$, $c_0>0$ and $c_1 > 0$ such that whenever
$x_0 \in \bdry{\set{u > 1}}$ and $0 < r \le r_0$,
there is $x_1\in \bdry{B_r(x_0)}$ such that
\[
u(x_1) - 1 \ge c_1 r
\]
Moreover,
\begin{equation} \label{65}
\fint_{\bdry{B_r(x_0)}} (u - 1)_+\, d\sigma  \ge c_0 r
\end{equation}
\end{lemma}
\begin{proof}
Suppose by contradiction that there is no such $c_1$.  Then
$(u(x)-1)_+ \ll r$ on $\partial B_r(x_0)$.  Note, in addition, that
\[
\Delta (u-1)_+ \ge -(u-1)_+^{p-1} \ge  - (L r)^{p-1} \quad \mbox{in} \quad B_r(x_0).
\]
Consider the barrier function $v$ solving $\Delta v = -(Lr)^{p-1}$ in
$B_r(x_0)$, with constant boundary values
$v = a \ll r$ on $\partial B_r(x_0)$.
Then $(u-1)_+ \le v$,  and for sufficiently small $r$, $v \ll r$
on all of $B_r(x_0)$.
But this contradicts Lemma \ref{Lemma 10} which says that there is a
point of $B_r(x_0) \setminus B_{r/2}(x_0)$ at which $u-1$ is larger
than $\gamma r$.

Next, take $x_1\in \bdry{B_r(x_0)}$ as above
for which $u(x_1)> 1+ c_1r$.  By Lipschitz continuity,
$u(x) > 1+ c_1r/2$ on $B_{c_1r/2L}(x_1)$.
Thus we have \eqref{65} for a constant $c_0>0$ depending only on $c_1$ and $L$.
\end{proof}


\begin{lemma} \label{Lemma 14}
There exist a  positive constants $c$, $\eps_0$ and $C$ depending on
dimension and the Lipschitz constant  $L$ such that
whenever $x_0 \in \bdry{\set{u > 1}}$, $0 < r \le r_0$,
\begin{equation} \label{64}
\sigma(\bdry{B_r(x_0)} \cap \set{u \le 1}) < \eps r^{N-1}
\end{equation}
for some $\eps < \eps_0$,
and $v$ is the harmonic function in $B_r(x_0)$ with $v = u$ on $\bdry{B_r(x_0)}$, we have
\begin{gather}
\label{62}
\int_{B_r(x_0)} (v - 1)_+^p\, dx +
 \int_{B_r(x_0)} (u - 1)_+^p\, dx   \le Cr^{p+N},\\[7.5pt]
\label{61} \int_{B_r(x_0)} |\nabla v|^2\, dx \le \int_{B_r(x_0)} |\nabla u|^2\, dx - cr^N,\\[7.5pt]
\label{63} \int_{\set{v \le 1}\cap B_r(x_0)} |\nabla v|^2\, dx
\le C \eps^{1/N} r^N.
\end{gather}
\end{lemma}

\begin{proof} We have $|u - 1| \le {L} r$ on $\closure{B_r(x_0)}$, and hence $|v-1|\le Lr$
by the maximum principle. Thus \eqref{62} follows.

To prove \eqref{61}, begin by noting that
\[
\int_{B_r(x_0)} \left(|\nabla u|^2 - |\nabla v|^2\right) dx = \int_{B_r(x_0)} |\nabla (u - v)|^2\, dx + 2 \int_{B_r(x_0)} \nabla (u - v) \cdot \nabla v\, dx.
\]
Since $u - v \in H^1_0(B_r(x_0))$,
\[
\int_{B_r(x_0)} |\nabla (u - v)|^2\, dx \ge \frac{\lambda_1}{r^2} \int_{B_r(x_0)} (u - v)^2\, dx,
\]
where $\lambda_1 > 0$ is the first Dirichlet eigenvalue of the negative Laplacian in $B_1(0)$, and since $u = v$ on $\bdry{B_r(x_0)}$ and $\Delta v = 0$ in $B_r(x_0)$,
an integration by parts gives
\[
\int_{B_r(x_0)} \nabla (u - v) \cdot \nabla v\, dx = \int_{\bdry{B_r(x_0)}} (u - v)\, \frac{\partial v}{\partial n} - \int_{B_r(x_0)} (u - v)\, \Delta v\, dx = 0
\]
Hence
\begin{equation} \label{60}
\int_{B_r(x_0)} \left(|\nabla u|^2 - |\nabla v|^2\right) dx \ge \frac{\lambda_1}{r^2} \int_{B_r(x_0)} (u - v)^2\, dx.
\end{equation}

We have
\begin{equation} \label{56}
|u(x) - v(x)| \ge |v(x_0) - u(x_0)| - |u(x) - u(x_0)| - |v(x) - v(x_0)|.
\end{equation}
Fix $\kappa \in (0,1/2)$.
Furthermore, for
\begin{equation} \label{58}
|u(x) - u(x_0)| \le L |x - x_0| \le L \kappa r, \quad x \in B_{\kappa r}(x_0)
\end{equation}
Since $|v-1| \le Lr$ on $B_r(x_0)$ and
$v$ is harmonic inside the ball, it follows that
$|\nabla v|\le C$ on $B_{\kappa r}(x_0)$, and hence
\begin{equation} \label{59}
|v(x) - v(x_0)| \le C \kappa r \quad \mbox{for all} \ x\in B_{\kappa r}(x_0).
\end{equation}
Since $u(x_0)=1$, the mean value property of $v$ implies
\begin{equation}\label{57}
\begin{aligned}
v(x_0) - u(x_0)
& = \fint_{\bdry{B_r(x_0)}} (u - 1)\, d\sigma
= \fint_{\bdry{B_r(x_0)}} (u - 1)_+\, d\sigma
- \fint_{\bdry{B_r(x_0)}} (u - 1)_-\, d\sigma \\
& \ge \left(c_0 - \frac{L \eps}{\omega_N}\right) r \ge \frac{c_0}{2}\, r
\end{aligned}
\end{equation}
by \eqref{65}, \eqref{64}, and since $|u - 1| \le Lr$ on $\bdry{B_r(x_0)}$.
Combining \eqref{56}--\eqref{57} and taking $\kappa$ sufficiently small
gives $|u(x) - v(x)| \ge c_0 r/3$, for all $x\in B_{\kappa r}(x_0)$,
Together with \eqref{60}, this yields \eqref{61}.

To prove \eqref{63}, we apply Lemma \ref{Lemma 13} to
\[
h(y) = \frac{1}{r}\, (v(x_0 + ry) - 1), \quad y \in B,
\]
noting that
\[
h(0) = \frac{1}{r}\, (v(x_0) - u(x_0)) \ge \frac{c_0}{2}
\]
by \eqref{57}.
\end{proof}


\begin{proof}[Proof of Proposition \ref{Theorem 3}]
Let $r_0$ and $\gamma > 0$ be as in Lemma \ref{Lemma 10},
let $x_0 \in \bdry{\set{u > 1}}$, and let $0 < r \le r_0$.
Then there is $ x_1 \in B_{r/2}(x_0)$ such that $u(x_1) \ge 1 + \gamma r/2$.
Let $\kappa = \min \bgset{1/2,\gamma/2 L}$. Then
\[
u(x) \ge u(x_1) - L |x - x_1| > 1 + \left(\frac{\gamma}{2} - L \kappa\right) r
\ge 1 \quad \forall x \in B_{\kappa r}(x_1),
\]
so the volume fraction of $\set{u>1}$ in $B_r(x_0)$ of \eqref{55}
is at least $\kappa^N$.

If the second inequality in \eqref{55} does not hold, then for arbitrarily
small $\rho, \gamma > 0$, $\exists\, x_0 \in \bdry{\set{u > 1}}$ such that
\begin{equation} \label{68}
\vol{\set{u \le 1} \cap B_\rho(x_0)} < \gamma \rho^N.
\end{equation}
Then
\[
\int_{\rho/2}^\rho \sigma(\set{u\le 1} \cap \partial B_r(x_0))\, dr
< \gamma \rho^N,
\]
and hence for some $r$, $\rho/2 \le r \le \rho$,
\[
\sigma(\set{u\le 1} \cap \partial B_r(x_0))
\le 2\gamma \rho^{N-1} \le 2^N \gamma r^{N-1}
\]
In other words, inequality \eqref{64} in Lemma \ref{Lemma 14} holds for
$\eps = 2^N \gamma$ and some $r \in (\rho/2,\rho)$.

Let $v$ be as in Lemma \ref{Lemma 14}, and let $w = v$
in $B_r(x_0)$ and $w = u$ in $\Omega \setminus B_r(x_0)$. Then
\begin{gather}
\label{66} \int_\Omega |\nabla w|^2\, dx \le \int_\Omega |\nabla u|^2\, dx - cr^N,\\[7.5pt]
\label{67} \abs{\int_{\set{w > 1}} (w - 1)^p\, dx - \int_{\set{u > 1}} (u - 1)^p\, dx} \le Cr^{p+N},\\[7.5pt]
\label{72} \int_{\set{w \le 1}} |\nabla w|^2\, dx \le \int_{\set{u \le 1}} |\nabla u|^2\, dx + C \eps^{1/N} r^N
\end{gather}
by \eqref{61}, \eqref{62}, and \eqref{63}, respectively.

By \eqref{68},
\begin{align*}
\L(\set{w>1}) & \le
\L(\set{u>1}\setminus B_r(x_0)) + \L(B_r(x_0)) \\
& = \L(\set{u>1}) + \L (\set{u\le  1} \cap B_r(x_0)) \\
& \le \L(\set{u>1}) + \gamma \rho^N,
\end{align*}
Recalling $r \ge \rho/2$ and $\eps = 2^N\gamma$, we have
\begin{equation} \label{73}
\vol{\set{w > 1}} \le \vol{\set{u > 1}} + \eps r^N.
\end{equation}

Estimate \eqref{68} also implies
\[
\int_{\set{u \le 1} \cap B_r(x_0)} |\nabla u|^2\, dx \le L^2 \gamma \rho^N
\]
which together with \eqref{66} gives
\begin{equation} \label{70}
\int_{\set{w > 1}} |\nabla w|^2\, dx \le \int_{\set{u > 1}} |\nabla u|^2\, dx - \frac{c}{2}\, r^N
\end{equation}
for sufficiently small $\gamma$.

Referring to \eqref{71}, by \eqref{70} and \eqref{67},
\[
s_w \le \left[\frac{\dint_{\set{u > 1}} |\nabla u|^2\, dx - \frac{c}{2}\, r^N}{\dint_{\set{u > 1}} (u - 1)^p\, dx - Cr^{p+N}}\right]^{1/(p-2)} \le 1
\]
for sufficiently small $r$ since $u \in \M$. Then by Lemma \ref{Lemma 8},
\[
J(\pi(w)) \le \frac12 \int_{\set{w<1}} |\nabla w|^2 \, dx
+ \left(\frac12 - \frac1p\right) \int_{\set{w>1}} |\nabla w|^2  \, dx
+ \L(\set{w>1})
\]
Finally, using \eqref{72}, \eqref{73}, and \eqref{70},
\[
J(\pi(w)) \le J(u) + \left[C \eps^{1/N} + \eps - \left(\frac{1}{2} - \frac{1}{p}\right) \frac{c}{2}\right] r^N < J(u)
\]
if $\eps$ is sufficiently small.   This is a contradiction, since $\pi(w) \in \M$
and $u$ minimizes $\restr{J}{\M}$.
\end{proof}

\section{Proof of regularity of the free boundary}

The purpose of this section is to prove Theorem \ref{thm:corollary}.
We do this by taking blow-up limits and applying a monotonicity result of
G. Weiss.

Consider a boundary point $x_0\in F(u)= \partial \set{u>1}$.
The Lipschitz continuity of $u$ implies there is
a sequence $r_j\to 0$ such that
\[
w_j(y) = r_j^{-1} (u(x_0 + r_jy) -1)
\]
converges uniformly on compact subsets of $\R^N$ to a Lipschitz continuous
function $W(y)$.   We now show that $W$ inherits all the properties we
found for $u$.
\begin{lemma} \label{lem:limitW}

a) The function $W$ is Lipschitz continuous, uniformly in all
$\R^N$, and solves the interior Euler-Lagrange equation
\[
\Delta W = 0 \quad \mbox{in} \quad \R^N \setminus \bdry \set{W>0}.
\]

b) (nondegeneracy of $W$) There is $c>0$ such that for every $r>0$
and every $y_1$ such that $B_r(y_1)\subset \set{W>0}$ we have
\[
W(y_1) \ge cr.
\]
For every $r>0$ and every $y_0\in \bdry \set{W>0}$
there is $y_1\in B_r(y_0)$ such that
\[
W(y_1) \ge cr.
\]

c) (locally finite perimeter)
There is a constant $C$ such that for every ball $B_r$ of radius $r>0$,
\[
\sigma(B_r\cap \bdry \set{W>0}) \le C r^{N-1}.
\]

d) (nondegeneracy of the phase $\set{W\le 0}$)
For every $r>0$ and every $y_0\in \bdry \set{W>0}$,
\[
\L(B_r(y_0) \cap \interior{\set{W\le0}}) \ge cr^N, \quad
\L(B_r(y_0) \cap \set{W>0}) \ge cr^N.
\]

e) (viscosity solution)
For every $r>0$, if there is a tangent ball from either
side of the free boundary, that is, a ball $B_r$ such that
$y_0 \in \partial B_r \cap \bdry \set{W>0}$  and either
$B_r  \subset \set{W>0}$ or $B_r \subset \set{W \le 0}$, then
$W$ has an asymptotic expansion as $y\to y_0$ of the form
\[
W(y) =  \alpha \langle y-y_0, \nu\rangle_+
- \beta \langle y-y_0, \nu\rangle_- + o(|y-y_0|),
\]
with $\alpha>0$, $\beta \ge 0$ and $\alpha^2 - \beta^2 = 2$.

f) (variational solution) $W$ satisfies the variational equation
\[
\int_{\R^N}
\left[\left( \frac12 |\nabla W|^2
+ \chi_{\set{W>1}}\right) \mbox{\rm div}\, \Phi
- \nabla w(D\Phi)\cdot \nabla w \right] \, dx = 0,
\]
for every $\Phi\in C_0^\infty(\R^N,\R^N)$.
\end{lemma}
\begin{proof}  All of the results except part (f)
are proved by methods of Caffarelli described in
\cite{MR861482, MR990856, MR1029856} and \cite{MR2145284}.
Part (f) is proved the same way as \cite{Jerison-Kamburov} Proposition 4.2.

On any compact subset of $\set{W>0}$, we have $w_j>0$ for
all sufficiently large $j$, and therefore $W$ inherits
the first nondegeneracy property of part (b) from $w_j$.
Moreover, the equation $\Delta w_j = r_j (w_j)_+^{p-1}$
holds in a fixed neighborhood of the compact set.
It follows that $w_j$ belongs to $C^{2,\alpha}$ uniformly
on the compact set.  Hence a subsequence of $w_j$ converges
is $C^2$ to $W$.  Taking the limit in the equation we
find that $\Delta W = 0$ on $\set {W>0}$.
The second nondegeneracy property of (b) follows from the first using
the Lipschitz bound and the fact that $W$ is harmonic
in the set $\set{W>0}$.

Denote
\[
E_j = \overline{ \set{w_j>1}}, \quad E = \overline{\set W>1}.
\]
We claim that for a suitable subsequence
\begin{equation}\label{eq:Hausdorff.conv}
E_j \cap \bar B \to E \cap \bar B
\end{equation}
in Hausdorff distance for every ball $B\subset \R^N$.
Choose the subsequence so that $E_j\cap \bar B_R$ converges
in Hausdorff distance to a compact set $K$.
We wish to show that
\[
K = E \cap \bar B
\]
Indeed, the fact that $w_j$ converges uniformly to $W$
implies $K \supset \set{W>0} \cap B$ and hence,
since $K$ is compact, $K \supset E \cap \bar B$.

If $x\notin E$, then we now show that for
sufficiently small $\eps>0$ and large enough $j$,
\begin{equation}\label{eq:wj0}
w_j(y) \le 0 \quad \mbox{for all} \ y \in B_\eps(x).
\end{equation}
Choose $\eps>0$ sufficiently small that
\[
B_{2\eps}(x) \cap E = \emptyset.
\]
Choose $\delta  \ll \eps$.  Since $W\le 0$ on $B_{2\eps}(x)$
and $w_j$ tends uniformly to $W$, for sufficiently large $j$,
\[
w_j(y) \le \delta \quad \mbox{for all} \ y \in B_{2\eps}(x)
\]
Suppose by contradiction that there is $y_1\in B_{\eps}(x)$ such that
$w_j(y_1) >0$.  By nondegeneracy (Definition \ref{def:nondegenerate})
$w_j(y_1) \le \delta$ implies there is $y_2$, $|y_2-y_1| \le C\delta$
such that $w_j(y_2) \le 0$.  Hence there is a point $y_3\in \bdry{w_j>0}$
on the segment between $y_1$ and $y_2$.
By the second form of nondegeneracy, Lemma \ref{Lemma 10}, there
is a point  $y_4\in B_{\eps}(y_3)$ for which $w_j(y_4) \ge \gamma \eps$.
But $y_4\in B_{2\eps}(x)$, so this contradicts $\delta << \eps$.

We have just shown in \eqref{eq:wj0} that for
all sufficiently large $j$, $B_\eps(x) \cap E_j = \emptyset$.
It follows that $x \notin K$, which finishes
the proof of \eqref{eq:Hausdorff.conv}.

Next, note that the same argument says that on compact
subsets of $E^c$ (the interior of $\set{W\le 0}$) we
have $w_j\le 0$ for sufficiently large $j$ and we
can use the equation $\Delta w_j = 0$ to conclude
that $\Delta W= 0$ on $E^c$.  This concludes part (a).

By the same argument as Proposition \ref{Theorem 2}, we have
part (c).  In particular,
\[
\L(\partial E) = \L(\partial \set{W>0}) =  0.
\]
By uniform convergence of $w_j$ to $W$ we have
\[
\L((E \setminus{w_j>0})\cap B)
= \L((\set{W >0} \setminus \set{w_j>0})\cap B  ) \to 0, \quad j\to \infty.
\]
for every ball $B$.  On the other hand, we just showed
that on every compact subset of $\interior{\set{W\le 0}}$,
we have $w_j\le 0$ for sufficiently large $j$.  From
this and the fact that $\bdry \set{W\le 0}$ has zero measure
it follows that
\[
\L((\set{W \le 0} \setminus \set{w_j\le 0})\cap B  ) \to 0, \quad j\to \infty.
\]
In all,
\[
\chi_{\set{w_j>0}} \to \chi_{\set{W>0}}  \quad \mbox{in} \quad L^1(B)
\]
Part (d) now follows from the convergence in Hausdorff
distance and the corresponding estimates for
$u$ in Proposition \ref{Theorem 3}.

Next we turn to part (e). It follows from the methods
of Caffarelli \cite{MR973745}, Caffarelli-Salsa \cite{MR2145284}, and of Lederman-Wolanski \cite{MR2281453}
that the limit $W$ is a solution in the weak viscosity
sense of Definition \ref{def:viscosity}. Moreover,
if there is a tangent ball at $y_0$ from either the
$\set{W>0}$ of the $\interior{\set{W\le 0}}$ side,
then $W$ has an asymptotic of the form
\[
W(y) =  \alpha \langle y-y_0, \nu\rangle_+
- \beta \langle y-y_0, \nu\rangle_- + o(|y-y_0|),
\]
with $\alpha>0$.  From part (d),
we have the additional information that
$\L(\set{W\le 0} \cap B_r(y_0)) \ge c r^N$, which rules
our the case $\beta<0$.  Thus $\beta\ge0$,
in that case, and the methods of Caffarelli also show that
$\alpha^2 -\beta^2 = 2$.

Finally, we demonstrate part (f) by using the variational
equation for $w_j$ and applying the dominated
convergence theorem.  Recall that if $K$ is
a compact subset of $\set{W>0}$, respectively, $\interior{\set{W\le 0}}$),
then for sufficiently large $j$, $K\subset \set{w_j>0}$, respectively
$K \subset \interior {\set{w_j \le 0}}$.  It follows that for
large $j$, $w_j$ is uniformly $C^{2,\alpha} $ on $K$.  Thus
taking subsequences, we may assume $\nabla w_j$ converges
pointwise to $\nabla W$ on $\R^N \setminus \bdry \set{W>0}$.
Since $\bdry  \set{W>0}$ has Lebesgue
measure zero, and the compact set $K$ was arbitrary,
we can choose the subsequence $\nabla w_j$ so that it
tends pointwise almost everywhere in $\R^N$
to $\nabla W$.  Recall also that on a suitable subsequence,
$\chi_{w_j>0} \to \chi_{W>0}$ in $L^1(B_r)$ for any $r<\infty$.
Since the test function $\Phi$ has compact support,
the dominated convergence theorem applies.
Taking the limit in the variational equation,
Definition \ref{def:variational}, for $w_j$, we obtain (f).

\end{proof}

The proof of Theorem \ref{thm:corollary} proceeds by induction on dimension.
The first step ($N=2$) requires relatively few of the
conclusions of Lemma \ref{lem:limitW}.

Consider any $x_0\in \partial \set{u>1}$ and $w_j\to W$ as above.
Because $w_j$ is a variational solution, the theorem of Weiss,
Corollary \ref{cor:weiss}, applies and says that the limit
$W$ is homogeneous, $W(ry) = rW(y)$
for all $y\in \R^N$.  By Lemma \ref{lem:limitW} (a), $W$ is harmonic
in the cones $\set{W>0}$ and $\interior{\set{W\le 0}}$.  When
$N=2$, the cones are sectors and by part (b)
of the lemma, the only possibility is that for some unit
vector $\nu$, and some $\alpha>0$ and $\beta \ge 0$,
\[
W(x) =  \alpha \langle x, \nu\rangle_+
- \beta \langle x, \nu\rangle_- \, .
\]
Incidentally, it does sometimes
happen that $W$ is strictly positive on both
sides of the free boundary for limits of
other kinds of non-minimizing critical points, even in dimension $2$
(see \cite{Jerison-Kamburov}).  But as in our earlier discussion
of viscosity solutions in Theorem \ref{thm:nondegeneracy},
$\beta<0$ is ruled out by the fact that
$\L(B_r(0) \cap \set{W\le 0}) \ge  cr^N$.

It follows from the central results in the work on free boundaries
of Caffarelli that
the free boundary of $u$ is a smooth hypersurface in
a neighborhood of $x_0$.  In fact, because $w_j$ tends
uniformly to $W$ and $w_j$ is nondegenerate,
the free boundary of $u$ is ``flat'' near $x_0$.  The solution
$u$ satisfies the free boundary condition in the viscosity
sense, and hence the free boundary is smooth using
the ``flat implies Lipschitz'' and ``Lipschitz implies smooth''
theorems of Caffarelli \cite{MR861482, MR990856, MR1029856};
see also \cite{MR2145284}.
(Those theorems were carried out for zero right hand side,
but can be modified to this situation without difficulty because
$(u-1)_+^{p-1}$ is zero at the free boundary.)
In dimension $3$, we follow the inductive method of G. Weiss.
Consider the cone
\[
\Gamma = \set{W>0}.
\]
Let $y_0\in \Gamma$, $y_0\neq0$.  By the uniform
Lipschitz bound on $W$, there is sequence $r_j\to 0$ for which
the limit
\[
\bar W(z): = \lim_{j\to \infty} r_j^{-1}W(y_0 + r_j z)
\]
exists and the convergence is uniform on compact subsets of $\R^N$.
Since the radial derivative
of $W$ is zero, one can show that $y_0 \cdot \nabla \bar W(z)  \equiv 0$
for all $z\in \R^N \setminus{\partial \set{\bar W>0}}$.  Thus
$\bar W$ is a two-dimensional solution.  Furthermore,
since by Lemma \ref{lem:limitW} (e), $W$ is a variational
solution, Corollary \ref{cor:weiss} with $Q\equiv 1$
implies that $\bar W$ is homogeneous.  It follows as
in the two-dimensional case that $\bar W$ is a planar solution
and hence that $\Gamma$ is smooth near $y_0$.  It then follows
that the free boundary of $u$ is flat near every point of a punctured
neighborhood of $x_0$.  The free boundary $\bdry \set{u>1}$ is
covered by finitely many balls of this type, and the free
boundary is smooth except possibly at the centers of these balls.
This completes the proof in dimension $3$.
The bound on the Hausdorff dimension of the singular set
in higher dimensions follows by an induction as in G. Weiss
\cite{MR1759450}.  This concludes the proof of Theorem \ref{thm:corollary}.

Without using scale-invariance and Weiss monotonicity
one can obtain a weaker, qualitative version of the preceding results.
Namely, the free boundary is smooth except
on a closed set of zero $(N - 1)$-dimensional Hausdorff measure.

Recall that Proposition \ref{Theorem 3} implies that the topological
boundary $\partial \set{u>1}$ is the same as the measure-theoretic
boundary.   Proposition \ref{Theorem 2}
implies that $\partial \set{u>1}$ has finite $(N-1)$-dimensional Hausdorff measure.
By the criterion for finite perimeter of Section 5.11 of Evans-Gariepy
\cite{MR1158660}, the set $\set{u>1}$ has finite perimeter, that is,
$\chi_{\set{u>1}}$ is a function of bounded variation.
By Lemma 1 Section 5.8 of Evans-Gariepy \cite{MR1158660},
the reduced boundary of a set of finite
perimeter is of full $(N-1)$-dimensional Hausdorff
measure in the measure-theoretic boundary.
Since, by definition, at every point of the reduced boundary
there is a measure-theoretic normal, we may apply
Theorem 9.2 of Lederman and Wolanski \cite{MR2281453})
saying that the free boundary is a $C^{1,\, \alpha}$-surface
in a neighborhood of each point for which there is a measure-theoretic normal.
(This theorem applies with the same hypotheses as the theorem
about viscosity solutions, namely, that $u_j$ tends uniformly
to $u$ and $u$ is nondegenerate.)  Thus the set of points
where the free boundary is smooth is an open set of full
$(N-1)$ Hausdorff measure in the free boundary.


\section{Appendix: Weiss monotonicity}

For completeness, we state and prove the monotonicity formula
of G. Weiss in the form used here.

\begin{proposition} \label{prop:weiss} [G. Weiss] \ Suppose that $w$ is a
Lipschitz continuous function in the unit ball $B\subset \R^N$.
Let $Q\in C^\alpha(B)$ for some $\alpha >0$ be such that $Q(0)=0$
Suppose that $w$ satisfies the variational free boundary equation
\[
\int_{B} \left[\frac12 |\nabla w|^2 + Q(x) \chi_{\set{w>0}} \right]
\mbox{\rm div}\,  \Phi\, dx
- \int_B \nabla w D\Phi \cdot \nabla w \, dx = 0
\]
for every $\Phi \in C^\infty_0(B, \R^N)$.
Suppose further that $w(0)=0$ and
$w\Delta w$ is well-defined as a distribution and satisfies
\[
|w \Delta w|  \le C|x|^\alpha.
\]
Denote
\[
\psi(r) = r^{-N} \int_{B_r} \left[ \frac12 |\nabla w|^2 + \chi_{\set{w>0}}\right]
\, dx -  \frac12 r^{-1-N} \int_{\partial B_r} w^2 \, d\sigma
\]
Then for every $0 < r_0 < r_1 \le 1$,
\[
\psi(r_1) - \psi(r_0) =  \int_{r_0 < |x| < r_1}
(x \cdot \nabla w - w)^2
\frac{dx}{|x|^{N+2}}  +  O(r_1^\alpha)
\]
\end{proposition}
\begin{proof}  The proof is close to the one in
G. Weiss \cite{MR1620644}.
Consider a test function $\Phi:\R^N \to \R^N$
that is Lipschitz continuous and compactly supported in $B$.
By taking convolution with a smooth approximate identity,
we can find a sequence of test functions in $C_0^\infty(B, \R^N)$
whose gradients tend pointwise almost everywhere to $\nabla \Phi$.
Thus, by the dominated convergence theorem, the variational equation
is valid for $\Phi$.

We will use the family of Lipschitz continuous test functions
$\Phi_\eps (x)  = \eta_\eps(x)x$, where
\[
\eta_\eps (x)
= \begin{cases}
1 & \quad |x| \le r-\eps \\
\frac{r-|x|}{\eps} & \quad r-\eps \le |x| \le r \\
0 & \quad r \le |x| \le 1.
\end{cases}
\]
We have
\[
\frac{\partial}{\partial x_j} \Phi^i_\eps(x) = \eta_\eps(x) \delta_{ij}
- \frac{x_ix_j}{\eps |x|} \chi_{\set{r-\eps < |x| < r}},
\]
and hence
\[
\mbox{\rm div}\,  \Phi_\eps(x) = N \eta_\eps(x)
- \frac{|x|}{\eps}\chi_{\set{r-\eps < |x| < r}}.
\]

By Fubini's theorem, for almost
every fixed $r$, $0 < r < 1$,
$\nabla w(ry)$, $y\in \partial B$, is a well-defined function in
$L^2(\partial B)$.
Furthermore, let $\phi\in L^\infty(\R)$ have compact
support and integral equal to $1$.  The vector-valued maximal theorem implies
that for almost every $r$,
\[
\lim_{\eps\to 0}  \frac1\eps \int\phi((s-r)/\eps) \nabla w(sy) \, ds
= \nabla w(ry)
\]
in $L^2(\partial B)$ norm.  Therefore, we can take the limit
as $\eps\to 0$ in the variational formula of the hypothesis
with test function $\Phi_\eps$ to find for almost every $r$,
\begin{align*}
 0 = \int_{B_r} &
\left[N\left(\frac12 |\nabla w|^2 + Q(x) \chi_{\set{w>0}} \right)  - |\nabla w|^2\right]
\, dx  \\
&
- r \int_{\bdry B_r}
\left(\frac12 |\nabla w|^2 + Q(x) \chi_{\set{w>0}} \right) \, d\sigma +
r \int_{\bdry B_r} \left(\frac{x\cdot\nabla w}{|x|}\right)^2 \,  d\sigma.
\end{align*}
Integrating by parts, and using $|w \Delta w| \le C|x|^\alpha$,
\[
\int_{B_r} |\nabla w|^2 \, dx =  \int_{\bdry B_r} w \frac{x\cdot \nabla w}{|x|}
\, d\sigma + O(r^{N+ \alpha})
\]
Combining these two equation, multiplying by  $-r^{-N-1}$ and using
$|Q-1| \le C|x|^\alpha$, we have
\begin{align*}
-N r^{-N-1} &
\int_{B_r}  \left(\frac12 |\nabla w|^2 + \chi_{\set{w>0}} \right) \, dx
+ r^{-N} \int_{\bdry B_r}
\left(\frac12 |\nabla w|^2 +  \chi_{\set{w>0}} \right) \, d\sigma \\
&= -
r^{-N-1}\int_{\bdry B_r} w \frac{x\cdot \nabla w}{r} \, d\sigma
 + r^{-N-2} \int_{\bdry B_r} (x\cdot \nabla w)^2 \,  d\sigma
 + O(r^{-1+ \alpha}).
\end{align*}

On the other hand, for almost every $r$, $0< r<1$,
\begin{align*}
\frac{d}{dr}& \left( r^{-2}\int_{y\in \bdry B} w(ry)^2 \, d\sigma\right) \\
&
= -2r^{-3} \int_{y\in \bdry B} w(ry)^2 \, d\sigma
+ r^{-2} \int_{y\in \bdry B} 2w(ry)\, y\cdot \nabla w(ry)\,  d\sigma  \\
&
= -2r^{-2-N} \int_{x\in \bdry B_r} w^2 \, d\sigma
+ 2 r^{-1-N} \int_{x\in \bdry B_r} w \frac{x\cdot \nabla w}{r}\,   d\sigma
\end{align*}
Thus, $\psi$ is absolutely continuous, and for almost
every $r$, $0<r<1$,
\begin{align*}
\psi'(r)
&  = -Nr^{-N-1} \int_{B_r}
\left(\frac12 |\nabla w|^2 + \chi_{w>0}\right) \, dx
+ r^{-N} \int_{\bdry B_r} \left(\frac12 |\nabla w|^2 + \chi_{w>0}\right)
\, d\sigma \\
& +r^{-2-N} \int_{\bdry B_r} w^2 \, d\sigma
-  r^{-1-N} \int_{\bdry B} w \frac{x\cdot \nabla w}{r} \,  d\sigma  \\
&
= r^{-N-2} \int_{\bdry B_r} (x\cdot \nabla w - w)^2\, d\sigma + O(r^{-1+\alpha}).
\end{align*}
Integrating in $r$ finishes the proof of the proposition.
\end{proof}

\begin{corollary} \label{cor:weiss}
Suppose that $w$ is as in Proposition \ref{prop:weiss}.  If $r_j$ tends to zero
and \[
\frac{1}{r_j} w(r_j x) \to W(x)
\]
uniformly on compact subsets of $\R^N$.  Then $W$ is homogeneous
of degree 1:
\[
W(rx) = rW(x)
\]
for all $x\in \R^N$.
\end{corollary}
\begin{proof} From the proposition, we have
\[
\int_{\set{|x|<1}} (x\cdot w - w)^2 \frac{dx}{|x|^{N+2}} < \infty.
\]
It follows that
\[
\int_{\set{|x|<r}} (x\cdot w - w)^2 \frac{dx}{|x|^{N+2}} \to 0 \quad \mbox{as}
\quad r\to 0.
\]
Hence, for fixed $0 < a < b < \infty$,  as $r_j\to 0$,
\[
\int_{a < |y|<b} (y\cdot \nabla w_j(y) - w_j(y))^2 \frac{dy}{|y|^{N+2}}
=
\int_{ar_j < |x|<br_j} (x\cdot \nabla w(x) - w(x))^2 \frac{dx}{|x|^{N+2}} \to 0.
\]
Thus the sequence $y\cdot\nabla w_j - w_j$ tends to zero in
$L^2$ norm on $a < |y|< b$.  A subsequence tends weakly to $y\cdot \nabla W - W$,
showing that $y\cdot \nabla  W - W = 0$ weakly in $L^2$.
In particular, $W$ is homogeneous of degree $1$ as a distribution
on $0 < |y| < \infty$.  Since $W$ is Lipschitz continuous, it
is also homogeneous in the ordinary sense.
\end{proof}

\def\cdprime{$''$}

\end{document}